\documentclass{amsart}
\usepackage{amssymb,mathtools}
\usepackage[british]{babel}
\usepackage{enumitem}
\usepackage[latin1]{inputenc}
\usepackage{url}
\usepackage{tikz-cd}

\newtheorem{theorem}{Theorem}
\numberwithin{theorem}{section}
\newtheorem{corollary}[theorem]{Corollary}
\newtheorem{lemma}[theorem]{Lemma}
\newtheorem{proposition}[theorem]{Proposition}
\theoremstyle{definition}
\newtheorem{definition}[theorem]{Definition}

\newtheorem{example}[theorem]{Example}

\newcommand{\supp}{\operatorname{supp}}
\newcommand{\Supp}{\operatorname{Supp}}
\newcommand{\tr}{\operatorname{Tr}}
\newcommand{\To}{\Rightarrow}
\newcommand{\rng}{\operatorname{rng}}
\newcommand{\en}{\operatorname{en}}
\newcommand{\nf}{=_{\operatorname{NF}}}
\newcommand{\rca}{\mathbf{RCA}}
\newcommand{\aca}{\mathbf{ACA}}

\newcommand{\lo}{\mathbf{LO}}
\newcommand{\kb}{\operatorname{KB}}
\newcommand{\len}{\operatorname{len}}

\title[Well ordering principles and $\Pi^1_4$-statements]{Well ordering principles and $\Pi^1_4$-statements:\\ a pilot study}
\author{Anton Freund}

\begin{document}

\begin{abstract}
In previous work, the author has shown that $\Pi^1_1$-induction along~$\mathbb N$ is equivalent to a suitable formalization of the statement that every normal function on the ordinals has a fixed point. More precisely, this was proved for a representation of normal functions in terms of J.-Y.~Girard's dilators, which are particularly uniform transformations of well orders. The present paper works on the next type level and considers uniform transformations of dilators, which are called $2$-ptykes. We show that $\Pi^1_2$-induction along~$\mathbb N$ is equivalent to the existence of fixed points for all $2$-ptykes that satisfy a certain normality condition. Beyond this specific result, the paper paves the way for the analysis of further $\Pi^1_4$-statements in terms of well ordering principles.
\end{abstract}


\maketitle

\section{Introduction}

A classical result of J.-Y.~Girard~\cite{girard87} and J.~Hirst~\cite{hirst94} shows that the following are equivalent over the usual base theory~$\rca_0$ of reverse mathematics (see~\cite{simpson09} for an introduction to the latter):
\begin{enumerate}[label=(\roman*)]
\item arithmetical comprehension (i.\,e.~the principal axiom of $\aca_0$),
\item when $(X,\leq_X)$ is a (countable) well order, then so is
\begin{equation*}
\omega^X:=\{\langle x_0,\dots,x_{n-1}\rangle\,|\,x_0,\dots,x_{n-1}\in X\text{ and }x_{n-1}\leq_X\dots\leq_X x_0\}
\end{equation*}
with the lexicographic order.
\end{enumerate}
If we think of $\langle x_0,\dots,x_{n-1}\rangle\in\omega^X$ as the Cantor normal form $\omega^{x_0}+\dots+\omega^{x_{n-1}}$, we see that $X\mapsto\omega^X$ represents the familiar operation from ordinal arithmetic. Note that the base theory $\rca_0$ proves that $\omega^X$ exists and is a linear order (whenever the same holds for~$X$). A~statement such as~(ii), which asserts that some computable transformation preserves well foundedness, is called a well ordering principle.

Many important principles of reverse mathematics (about iterated arithmetical comprehension~\cite{marcone-montalban,rathjen-weiermann-atr} and the existence of $\omega$-models~\cite{rathjen-atr,rathjen-model-bi,thomson-thesis,thomson-rathjen-Pi-1-1}) have been characterized in terms of well ordering principles. At least in principle, there is no limitation on the consistency strength of statements that can be characterized in this way. There is, however, a limitation in terms of logical complexity: For any computable transformation of orders, the statement that this transformation preserves well foundedness has complexity~$\Pi^1_2$. Hence a genuine $\Pi^1_3$-statement---such as the principle of $\Pi^1_1$-comprehension---cannot be equivalent to a well ordering principle of this form.

In order to characterize statements of higher logical complexity, one needs to consider transformations of higher type. Recall that the usual notion of computability on the natural numbers extends to higher types if one restricts to functionals that are continuous in a suitable sense. In the present context, the continuous transformations between well orders are the dilators of J.-Y.~Girard~\cite{girard-pi2} (see below for details). Girard has sketched a proof that $\Pi^1_1$-comprehension is equivalent to the statement that his functor~$\Lambda$ preserves dilators~\cite{girard-book-part2}. As far as the present author is aware, the details of this proof have not been worked out. However, the present author~\cite{freund-equivalence,freund-categorical,freund-computable} has shown that $\Pi^1_1$-comprehension is equivalent to the following different well ordering principle: For every dilator~$D$ there is a well order~$X$ and a collapse $\vartheta:D(X)\to X$ that is almost order preserving in a suitable sense. Here the order~$X$ and the function~$\vartheta$ can be computed from a representation of~$D$. Only the fact that $X$ is well founded cannot be proved in~$\rca_0$.

Well ordering principles are relevant because they allow to apply methods from ordinal analysis to questions of reverse mathematics. Together with M.~Rathjen and A.~Weiermann, the present author has used his result on a collapse $\vartheta:D(X)\to X$ to show that $\Pi^1_1$-comprehension is equivalent to a uniform version of Kruskal's tree theorem~\cite{frw-kruskal}. This implies that the uniform Kruskal theorem exhausts the full strength of C.~Nash-Williams's famous ``minimal bad sequence argument"~\cite{nash-williams63}, in contrast to the usual Kruskal theorem. The present author has also shown that iterated applications of the uniform Kruskal theorem yield a systematic reconstruction of H.~Friedman's gap condition~\cite{freund-kruskal-gap}.

Much of reverse mathematics concerns statements of complexity $\Pi^1_2$ and~$\Pi^1_3$. In some applications, however, $\Pi^1_4$-statements play a central role. An example is the principle of~$\Pi^1_2$-bar induction, which is used in the proof of the graph minor theorem~\cite{krombholz-rathjen-GM}. The strongest ordinal analysis due to Rathjen~\cite{rathjen-Pi12} is also concerned with a $\Pi^1_4$-statement, namely with the principle of $\Pi^1_2$-comprehension. For these reasons it seems particularly relevant to extend well ordering principles to the level of $\Pi^1_4$-statements, which requires another step in the type structure. As far as the author is aware, the present paper is the first to make this step. Since the paper is intended as a pilot study, we will consider a particularly simple $\Pi^1_4$-statement, namely the principle of $\Pi^1_2$-induction along~$\mathbb N$. Even though the latter was chosen for its simplicity, it is of independent foundational interest: over $\mathbf{\Pi^1_1}\textbf{-CA}_0$, the principle of $\Pi^1_2$-induction along~$\mathbb N$ is needed to show that the finite trees with Friedman's gap condition form a well quasi order~\cite{simpson85}.

To state our result, we need some terminology; full details of the following can be found in Sections~\ref{sect:cat-dilators},~\ref{sect:ptykes} and~\ref{sect:normality-ptykes}. A predilator is an endofunctor of linear orders that preserves direct limits, pullbacks and the order on morphisms (pointwise domination). If it preserves well foundedness, then it is called a dilator. By a morphism between (pre-)dilators $D$ and~$E$ we mean a natural transformation $\mu:D\To E$ of functors. Such a transformation consists of an order embedding $\mu_X:D(X)\to E(X)$ for each linear order~$X$. If the image of each embedding $\mu_X$ is an initial segment of the order~$E(X)$, then we call $\mu$ a segment. A $2$-preptyx is an endofunctor of predilators that preserves direct limits and pullbacks. If it preserves dilators, then it is called a $2$-ptyx (plural: ptykes). The number~$2$ indicates the type level, where we think of well orders and dilators as objects of type zero and one, respectively; it will be omitted where the context allows it. We say that a $2$-(pre-)ptyx~$P$ is normal if~$P(\mu):P(D)\To P(E)$ is a segment for any segment~$\mu:D\To E$. Preservation of direct limits is a continuity property, which allows us to represent predilators and preptykes in second order arithmetic. Based on such a representation, we will show that the following are equivalent over~$\aca_0$ (cf.~Theorem~\ref{thm:main} below):
\begin{enumerate}[label=(\roman*)]
\item the principle of $\Pi^1_2$-induction along~$\mathbb N$,
\item for any normal $2$-ptyx~$P$ there is a dilator~$D\cong P(D)$.
\end{enumerate}
Since~$P$ preserves direct limits, a predilator~$D_P\cong P(D_P)$ can be constructed as the direct limit of the diagram
\begin{equation*}
D^0_P\xRightarrow{\mu^0}D^1_P:=P(D^0_P)\xRightarrow{\mu^1:=P(\mu^0)}D^2_P:=P(D^1_P)\xRightarrow{\mu^2:=P(\mu^1)}\ldots,
\end{equation*}
where $D^0_P$ is the constant dilator with value~$0$ (the empty order). In Section~\ref{sect:construct-fixed-points} we will see that this construction can be implemented in~$\aca_0$ (and presumably also in~$\rca_0$). Hence the strength of~(ii) does not lie in the existence of~$D\cong P(D)$, but rather in the assertion that~$D$ is a dilator (i.\,e., preserves well foundedness). In order to prove that~(i) implies~(ii), one uses~$\Pi^1_2$-induction to show that each predilator~$D^n_P$ in the above diagram is a dilator. The assumption that~$P$ is normal is needed to conclude that the same holds for the direct limit (by Example~\ref{ex:ptyx-succ}, which shows that $P(D):=(D+1)(X):=D(X)+1$ defines a $2$-ptyx that is not normal and cannot have a dilator as a fixed point).

In order to explain the proof that~(ii) implies~(i), we recall another result on the level of~$\Pi^1_3$-statements. A function~$f$ on the ordinals is called normal if it is strictly increasing and continuous at limits (i.\,e.,~we demand $f(\lambda)=\sup_{\alpha<\lambda}f(\alpha)$ for any limit ordinal~$\lambda$). It is well known that any normal function~$f$ has a proper class of fixed points. These are enumerated by another normal function~$f'$, which is called the derivative of~$f$. Together with M.~Rathjen, the present author has shown that $\Pi^1_1$-bar induction is equivalent to a suitable formalization of the statement that every normal function has a derivative~\cite{freund-rathjen_derivatives,freund-ordinal-exponentiation}. Furthermore, $\Pi^1_1$-induction along~$\mathbb N$ is equivalent to the statement that each normal function has at least one fixed point~\cite{freund-single-fixed-point}. In the following, we recall the intuition behind this result.

Given a linear order~$X$ and an ordinal~$\alpha$, we write $X\preceq\alpha$ to express that~$X$ is well founded with order type at most~$\alpha$. On an intuitive level, the fact that well foundedness is~$\Pi^1_1$-complete allows us to write any instance of $\Pi^1_1$-induction as
\begin{equation*}
X_0\preceq\alpha_0\land\forall_{n\in\mathbb N}(\exists_\alpha\, X_n\preceq\alpha\to\exists_\beta\,X_{n+1}\preceq\beta)\to\forall_{n\in\mathbb N}\exists_\gamma\, X_n\preceq\gamma,
\end{equation*}
for some family of linear orders~$X_n$. Assume that the second conjunct of the premise is witnessed by a function $h_0$ such that $X_n\preceq\alpha$ implies $X_{n+1}\preceq h_0(n,\alpha)$. If we set~$h(\alpha):=\sup_{n\in\mathbb N}h_0(n,\alpha)$, then we obtain
\begin{equation*}
\forall_{n\in\mathbb N}\forall_\alpha(X_n\preceq\alpha\to X_{n+1}\preceq h(\alpha)).
\end{equation*}
As~$h$ may not be normal, we consider the normal function~$g$ with
\begin{equation*}
g(\alpha)=\sum_{\gamma<\alpha}h(\gamma)+1.
\end{equation*}
More formally, this function can be defined by the recursive clauses $g(0)=0$, $g(\alpha+1)=g(\alpha)+h(\alpha)+1$ and $g(\lambda)=\sup_{\gamma<\lambda}g(\gamma)$ for $\lambda$ limit. We observe
\begin{equation*}
\gamma+1\leq\alpha\quad\Rightarrow\quad h(\gamma)+1\leq g(\alpha).
\end{equation*}
To incorporate the premise $X_0\preceq\alpha_0$ of our induction statement, we transform~$g$ into another normal function~$f$ with $f(\alpha):=\alpha_0+1+g(\alpha)$. If $\alpha=f(\alpha)$ is a fixed point, then a straightforward induction over~$n\in\mathbb N$ yields
\begin{equation*}
\forall_{n\in\mathbb N}\exists_\gamma(X_n\preceq\gamma\land\gamma+1\leq\alpha).
\end{equation*}
In particular we have $\forall_{n\in\mathbb N}X_n\preceq\alpha$, which entails the conclusion of the induction statement above. As shown in~\cite{freund-single-fixed-point}, the given argument can be formalized in terms of dilators. In that setting, the induction at the end of the argument is not needed: it can be replaced by a construction that builds embeddings $X_n+1\hookrightarrow\alpha$ by primitive recursion over elements of~$X_n$, simultaneously for all~$n\in\mathbb N$. The formalized argument deduces $\Pi^1_1$-induction along~$\mathbb N$ from the assumption that every normal function (represented by a dilator) has a fixed point.

In order to deduce $\Pi^1_2$-induction, we will lift the previous argument to the next type level. This relies, first of all, on Girard's result~\cite{girard-book-part2} that the notion of dilator is $\Pi^1_2$-complete. Given an arbitrary $\Pi^1_2$-formula~$\psi$, one can thus construct a family of pre\-dilators~$D^n_\psi$ such that induction for~$\psi$ is equivalent to the following statement: If $D^0_\psi$ is a dilator and $D^{n+1}_\psi$ is a dilator whenever the same holds for~$D^n_\psi$, then~$D^n_\psi$ is a dilator for every~$n\in\mathbb N$. Assuming the premise of this implication, we will be able to construct a $2$-ptyx~$P$ that admits a morphism
\begin{equation*}
D^{n+1}_\psi\To P(D^n_\psi)
\end{equation*}
for each~$n\in\mathbb N$. Note that $P$ corresponds to the function~$h$ from the argument above. Next, we need to transform~$P$ into a normal $2$-ptyx~$P^*$ that corresponds to the normal function~$g$. In our opinion, it is somewhat surprising that this is possible: The construction of~$g$ relies on the fact that each~$\alpha$ has a well ordered set of predecessors. A priori, this fact seems specific to the ordinals. However, Girard has discovered an analogous result on the next type level: Let us temporarily write $D\ll E$ to indicate that there is a segment $\mu:D\To E$ that is not an isomorphism. If $E$ is a predilator, then the isomorphism classes of predilators~$D\ll E$ form a set on which $\ll$ is linear (see~\cite[Lemma~2.11]{girard-normann85}). It is straightforward to define a pointwise sum of predilators. On an informal level, this allows us to set
\begin{equation*}
P^*(E):=\sum_{D\ll E}P(D)+1,
\end{equation*}
where $1$ refers to the constant dilator with that value. In view of $D\ll D+1$, we obtain a morphism
\begin{equation*}
P(D)+1\To P^*(D+1)
\end{equation*}
for each predilator~$D$. This completes the reconstruction of $g$ on the next type level. To define a normal $2$-ptyx $P^+$ that corresponds to the normal function~$f$, it suffices to set $P^+(E):=D_0+1+P^*(E)$. Given a dilator $E\cong P^+(E)$, one can construct morphisms $D_n+1\To E$ by (effective) recursion on~$n\in\mathbb N$, as on the previous type level. These ensure that the predilators $D_n$ are dilators, which is the conclusion of $\Pi^1_2$-induction. The given argument is made precise in Sections~\ref{sect:normality-ptykes} and~\ref{sect:fixed-point-to-induction}.

\section{The category of dilators}\label{sect:cat-dilators}

In this section we recall the definition and basic theory of dilators, both of which are due to J.-Y.~Girard~\cite{girard-pi2}. As the latter has observed, the continuity properties of dilators allow to represent them in second order arithmetic. Details of such a representation have been worked out in~\cite[Section~2]{freund-computable} and will also be recalled. Even though the material is known, this section plays a crucial role in the context of our paper: it fixes an efficient formalism upon which we can base our investigation of ptykes. The section also ensures that our paper is reasonably self contained.

Let $\lo$ be the category of linear orders, with the order embeddings (strictly increasing functions) as morphisms. For morphisms $f,g:X\to Y$ we abbreviate
\begin{equation*}
f\leq g\quad:\Leftrightarrow\quad f(x)\leq_Y g(x)\text{ for all $x\in X$}.
\end{equation*}
We say that a functor~$D:\lo\to\lo$ is monotone if $f\leq g$ implies $D(f)\leq D(g)$. The forgetful functor to the underlying set of an order will be omitted; conversely, subsets of the underlying set will often be considered as suborders. Given a set $X$, we write $[X]^{<\omega}$ for the set of its finite subsets. In order to obtain a functor, we define $[f]^{<\omega}(a):=\{f(x)\,|\,x\in a\}\in[Y]^{<\omega}$ for $f:X\to Y$ and $a\in[X]^{<\omega}$. Let us also write $\rng(f):=\{f(x)\,|\,x\in X\}\subseteq Y$ for the range of a function $f:X\to Y$.

\begin{definition}\label{def:dilator}
A predilator consists of
\begin{enumerate}[label=(\roman*)]
\item a monotone functor $D:\lo\to\lo$ and
\item a natural transformation $\supp:D\To[\cdot]^{<\omega}$ such that we have
\begin{equation*}
\{\sigma\in D(Y)\,|\,\supp_Y(\sigma)\subseteq\rng(f)\}\subseteq\rng(D(f))
\end{equation*}
for any morphism $f:X\to Y$.
\end{enumerate}
If $D(X)$ is well founded for every well order~$X$, then $D=(D,\supp)$ is a dilator.
\end{definition}

To give an example, we point out that the transformation $X\mapsto\omega^X=:D(X)$ from the introduction becomes a dilator if we set
\begin{align*}
D(f)(\langle x_0,\dots,x_{n-1}\rangle)&:=\langle f(x_0),\dots,f(x_{n-1})\rangle,\\
\supp_X(\langle x_0,\dots,x_{n-1}\rangle)&:=\{x_0,\dots,x_{n-1}\}.
\end{align*}
The inclusion in part~(ii) of Definition~\ref{def:dilator} will be called the support condition. Note that its converse is automatic: for $\sigma=D(f)(\sigma_0)\in\rng(D(f))$, naturality entails
\begin{equation*}
\supp_Y(\sigma)=\supp_Y(D(f)(\sigma_0))=[f]^{<\omega}(\supp_X(\sigma_0))\subseteq\rng(f).
\end{equation*}
If we take $f$ to be an inclusion $\iota_b^Y:X=b\hookrightarrow Y$ with $\rng(\iota_b^Y)=b\subseteq Y$, then we see that $\supp_Y(\sigma)$ is determined as the smallest set $b\in[Y]^{<\omega}$ with $\sigma\in\rng(D(\iota_b^Y))$. The fact that there is always a smallest finite set with this property implies that $D$ preserves direct limits and pullbacks; conversely, if $D:\lo\to\lo$ preserves direct limits and pullbacks, there is a unique natural transformation $\supp:D\To[\cdot]^{<\omega}$ that satisfies the support condition (essentially by Girard's normal form theorem~\cite{girard-pi2}; see~also~\cite[Remark~2.2.2]{freund-thesis} and the proof of Proposition~\ref{prop:ptykes-direct-limit} below). This shows that the given definition of dilator is equivalent to the original one by Girard. The condition that $D$ must be monotone is automatic when $X\mapsto D(X)$ preserves well foundedness (by~\cite[Proposition~2.3.10]{girard-pi2}; see~\cite[Lemma~5.3]{frw-kruskal} for a proof that uses our terminology). It has been omitted in previous work by the present author but will be important for this paper (see the proof of Lemma~\ref{lem:linearize-ll}).

The present paper considers dilators in second order arithmetic. When we speak of predilators in the sense of Definition~\ref{def:dilator}, we will always assume that they are given by $\Delta^0_1$-definitions of the relations
\begin{equation*}
\sigma\in D(X),\quad\sigma<_{D(X)}\tau,\quad D(f)(\sigma)=\tau,\quad\supp_X(\sigma)=a.
\end{equation*}
Here $\sigma,\tau$ and $a$ (which codes a finite set) range over natural numbers, while $X$ and $f:X\to Y$ are represented by subsets of~$\mathbb N$. In particular, this means that we interpret $\lo$ as the category of countable linear orders (with underlying sets contained in $\mathbb N$). We can use number and set parameters to quantify over families of predilators. In fact, we will see that there is a single $\Delta^0_1$-definable family that is universal in the sense that any predilator is isomorphic to one in this family. Our universal family will be parameterized by subsets of~$\mathbb N$, which are called coded predilators. The existence of such a universal family will be essential for our approach to ptykes, which are supposed to take arbitrary predilators as arguments. As explained after Theorem~\ref{thm:class-coded-equiv} below, the restriction to $\Delta^0_1$-definable predilators is inessential in a certain sense.

In order to construct a universal family of predilators, one exploits the fact that these are essentially determined by their restrictions to a small category~$\lo_0$. The objects of $\lo_0$ are the finite sets $\{0,\dots,n-1\}=:n$ with the usual linear order; the morphisms are the strictly increasing functions between them.

\begin{definition}\label{def:coded-predil}
A coded predilator consists of a monotone functor $D:\lo_0\To\lo$ and a natural transformation $\supp:D\To[\cdot]^{<\omega}$, such that the support condition from part~(ii) of Definition~\ref{def:dilator} is satisfied for all morphisms in~$\lo_0$.
\end{definition}

Recall that the underlying set of each linear order~$D(n)\in\lo$ is contained in~$\mathbb N$. Based on a suitable encoding of finite objects (such as morphisms $f:m\to n$ and sets $a\subseteq n=\{0,\dots,n-1\}$), we assume that coded predilators are given as sets
\begin{multline*}
D=\{(0,n,\sigma)\,|\,\sigma\in D(n)\}\cup\{(1,n,\sigma,\tau)\,|\,\sigma<_{D(n)}\tau\}\cup{}\\
{}\cup\{(2,f,\sigma,\tau)\,|\,D(f)(\sigma)=\tau\}\cup\{(3,n,\sigma,a)\,|\,\supp_n(\sigma)=a\}\subseteq\mathbb N.
\end{multline*}
For a coded dilator~$D$, this means that an expression such as $\sigma\in D(n)$ is an abbreviation for the $\Delta^0_1$-formula $(0,n,\sigma)\in D$. From now on we speak of class-sized predilators when we want to refer to predilators in the sense of Definition~\ref{def:dilator}. Recall that any class-sized dilator~$D$ is given by a $\Delta^0_1$-formula. Working over~$\rca_0$, this ensures that the obvious restriction~$D\!\restriction\!\lo_0$ exists as a set.

\begin{lemma}\label{lem:class-to-coded-predil}
If $D$ is a class-sized predilator, then $D\!\restriction\!\lo_0$ is a coded predilator.
\end{lemma}

Conversely, we will now show how a coded predilator can be extended into a class-sized one. Let us write $|a|$ for the cardinality of a finite order~$a$. As before, the order $\{0,\dots,|a|-1\}\in\lo_0$ will also be denoted by $|a|$.

\begin{definition}\label{def:trace}
The trace of a coded predilator~$D$ is given by
\begin{equation*}
\tr(D):=\{(n,\sigma)\,|\,\sigma\in D(n)\text{ and }\supp_n(\sigma)=n\}.
\end{equation*}
A class-sized predilator $D$ has trace $\tr(D):=\tr(D\!\restriction\!\lo_0)$.
\end{definition}

The equation $\supp_n(\sigma)=n$ in the definition of the trace is called the minimality condition: it states that the set $n=\{0,\dots,n-1\}$ is minimal in the sense that $\sigma$ depends on all its elements. The role of the minimality condition will become clear in the proof of Theorem~\ref{thm:class-coded-equiv}. Let us write $\en_a:|a|\to a$ for the increasing enumeration of a finite order~$a$. Given an embedding $f:a\to b$ between finite orders, let $|f|:|a|\to|b|$ be the unique morphism in~$\lo_0$ that satisfies
\begin{equation*}
\en_b\circ|f|=f\circ\en_a.
\end{equation*}
We also agree to write $\iota_X^Y:X\hookrightarrow Y$ for the inclusion map between sets~$X\subseteq Y$.

\begin{definition}\label{def:extend-coded-dil}
Consider a coded predilator~$D$. For each linear order~$X$ we set
\begin{gather*}
\overline D(X):=\{(a,\sigma)\,|\,a\in[X]^{<\omega}\text{ and }(|a|,\sigma)\in\tr(D)\},\\
(a,\sigma)<_{\overline D(X)}(b,\tau)\quad:\Leftrightarrow\quad D(|\iota_a^{a\cup b}|)(\sigma)<_{D(|a\cup b|)}D(|\iota_b^{a\cup b}|)(\tau),
\end{gather*}
where $a\cup b$ is ordered as a subset of~$X$. Given an embedding $f:X\to Y$, we define $\overline D(f):\overline D(X)\to\overline D(Y)$ by $\overline D(f)(a,\sigma):=([f]^{<\omega}(a),\sigma)$, relying on $|[f]^{<\omega}(a)|=|a|$. To define a family of functions $\overline\supp_X:\overline D(X)\to[X]^{<\omega}$, we set $\supp_X(a,\sigma):=a$.
\end{definition}

If $D$ is a coded predilator, then $D(|\iota_a^{a\cup b}|):D(|a|)\to D(|a\cup b|)$ is a total function. As a consequence, the relation $(a,\sigma)<_{\overline D(X)}(b,\tau)$ is $\Delta^0_1$-definable. It is even more straightforward to see that $(a,\sigma)\in\overline D(X)$, $\overline D(f)(a,\sigma)=(b,\tau)$ and $\overline\supp_X(a,\sigma)=b$ are $\Delta^0_1$-definable relations with parameter~$D\subseteq\mathbb N$, as needed for the following.

\begin{proposition}\label{prop:coded-to-class-predil}
If $D$ is a coded predilator, then~$\overline D$ is a class-sized predilator.
\end{proposition}
\begin{proof}
Except for the claim that the functor $\overline D$ is monotone, this has been verified in~\cite[Lemma~2.2]{freund-computable}. We point out that the cited reference works in a weak set theory; one readily checks that the same proof goes through in~$\rca_0$. The monotonicity of~$\overline D$ is established in~\cite[Lemma~5.2]{frw-kruskal}.
\end{proof}

Starting with a class-sized predilator~$D$, we can first form the restriction~$D\!\restriction\!\lo_0$ and then the extension according to Definition~\ref{def:extend-coded-dil}. In the following, we show that we essentially recover $D$ in this way.

\begin{definition}
Consider a class-sized dilator~$D$. For each order~$X$, we define a function $\eta^D_X:\overline{D\!\restriction\!\lo_0}(X)\to D(X)$ by setting
\begin{equation*}
\eta^D_X(a,\sigma):=D(\iota_a^X\circ\en_a)(\sigma),
\end{equation*}
for the enumeration $\en_a:|a|\to a$ and the inclusion $\iota_a^X:a\hookrightarrow X$.
\end{definition}

By our standing assumption on class-sized predilators, the relation $D(f)(\sigma)=\tau$ has a $\Delta^0_1$-definition, in which $f$ occurs as a set variable. To obtain a $\Delta^0_1$-definition of~$\eta^D$, one replaces each occurrence of $(k,x)\in f$ by the $\Delta^0_1$-formula $\en_a(k)=x$ (note that the finite enumeration $\en_a:|a|\to a\subseteq X$ can be computed with~$X$ as an oracle, and that we have $\iota_a^X(x)=x$).

\begin{theorem}\label{thm:class-coded-equiv}
If $D$ is a class-sized predilator, then $\eta^D:\overline{D\!\restriction\!\lo_0}\To D$ is a natural isomorphism of functors.
\end{theorem}
\begin{proof}
This result is proved in~\cite[Proposition~2.1]{freund-computable}. For later reference we recall the proof that each component $\eta_X$ is bijective. For $\sigma\in D(X)$ we write
\begin{equation*}
\sigma\nf D(\iota_a^X\circ\en_a)(\sigma_0)
\end{equation*}
with $a\in[X]^{<\omega}$ if the equation holds and we have $(|a|,\sigma_0)\in\tr(D)$. We need to show that each element of~$D(X)$ has a unique such normal form. To establish uniqueness, we observe that the minimality condition $\supp_{|a|}(\sigma_0)=|a|$ from the definition of the trace allows to recover
\begin{equation*}
a=[\iota_a^X\circ\en_a]^{<\omega}(\supp_{|a|}(\sigma_0))=\supp_X(D(\iota_a^X\circ\en_a)(\sigma_0)).
\end{equation*}
Once~$a$ is determined, uniqueness follows from the fact that $D(\iota_a^X\circ\en_a)$ is an embedding and hence injective. To establish existence, we set $a:=\supp_X(\sigma)$. Due to $\rng(\iota_a^X\circ\en_a)=a$, the support condition from part~(ii) of Definition~\ref{def:dilator} yields $\sigma=D(\iota_a^X\circ\en_a)(\sigma_0)$ for some $\sigma_0\in D(|a|)$. To show that we have a normal form, we need to establish~$(|a|,\sigma)\in\tr(D)$. In view of
\begin{equation*}
a=\supp_X(\sigma)=\supp_X(D(\iota_a^X\circ\en_a)(\sigma_0))=[\iota_a^X\circ\en_a]^{<\omega}(\supp_{|a|}(\sigma_0)),
\end{equation*}
the minimality condition $\supp_{|a|}(\sigma_0)=|a|$ must indeed hold.
\end{proof}

Let us also observe that the natural transformation $\eta^D$ respects the supports that come with $D$ and $\overline{D\!\restriction\!\lo_0}$. For $(a,\sigma)\in\overline D(X)$, the minimality condition entails
\begin{multline*}
\supp_X\circ\eta^D_X(a,\sigma)=\supp_X(D(\iota_a^X\circ\en_a)(\sigma))=[\iota_a^X\circ\en_a]^{<\omega}(\supp_{|a|}(\sigma))=\\
=[\iota_a^X\circ\en_a]^{<\omega}(|a|)=a=\overline\supp_X(a,\sigma).
\end{multline*}
This explicit verification will be superseded by the general result in Lemma~\ref{lem:transfos-respect-supp}. Theorem~\ref{thm:class-coded-equiv} shows that Definition~\ref{def:extend-coded-dil} yields a universal $\Delta^0_1$-definable family of class-sized predilators, in which coded predilators serve as set parameters. As promised, the previous considerations also show that the restriction to $\Delta^0_1$-definitions is inessential to a certain extent: It was only needed to ensure that the set $D\!\restriction\!\lo_0$ and the components $\eta_X:\overline{D\!\restriction\!\lo_0}(X)\to D(X)$ can be formed in~$\rca_0$. In a sufficiently strong base theory, the same argument shows that a class-sized predilator~$D$ of arbitrary complexity is equivalent to the $\Delta^0_1$-definable predilator $\overline{D\!\restriction\!\lo_0}$. Let us now come back to questions of well foundedness. We recall that $\rca_0$ proves the equivalence between the two obvious definitions (in terms of minimal elements and descending sequences; see e.\,g.~\cite[Lemma~2.3.12]{freund-thesis}).

\begin{definition}\label{def:coded-dilator}
Consider a coded predilator~$D$. If $\overline D(X)$ is well founded for every well order~$X$, then we say that $D$ is a coded dilator.
\end{definition}

From the viewpoint of second order arithmetic, the previous definition does only quantify over countable well orders. In a sufficiently strong setting, we can see that this does not make a difference: since each element $(a,\sigma)\in\overline D(X)$ has finite support~$a\in[X]^{<\omega}$, a descending sequence in~$\overline D(X)$ yields a descending sequence in~$\overline D(X_0)$ for some countable suborder~$X_0\subseteq X$. We can now complement Lemma~\ref{lem:class-to-coded-predil} and Proposition~\ref{prop:coded-to-class-predil} as follows:

\begin{corollary}\label{cor:class-coded-dil}
If $D$ is a class-sized dilator, then $D\!\restriction\!\lo_0$ is a coded dilator. If $D$ is a coded dilator, then $\overline D$ is a class-sized dilator.
\end{corollary}
\begin{proof}
To establish the first claim, we assume that $D$ is a class-sized dilator. In view of Theorem~\ref{thm:class-coded-equiv} it follows that $\overline{D\!\restriction\!\lo_0}(X)\cong D(X)$ is well founded for each well order~$X$, as needed to conclude that $D\!\restriction\!\lo_0$ is a coded dilator. The second claim is an immediate consequence of the definitions.
\end{proof}

In view of the close connection that we have established, we will omit the specifications ``class-sized" and ``coded" when the context allows it. Many constructions and results apply---mutatis mutandis---to both class-sized and coded predilators.

To turn the collection of predilators into a category, we declare that the morphisms between two predilators $D=(D,\supp^D)$ and $E=(E,\supp^E)$ are the natural transformations $\mu:D\To E$ of functors. Note that the components of such a transformation are morphisms in~$\lo$, i.\,e., order embeddings. In the coded case, we assume that $\mu$ is given as the set $\{(n,\sigma,\tau)\,|\,\mu_n(\sigma)=\tau\}\subseteq\mathbb N$; in the class-sized case, we require that the relation $\mu_X(\sigma)=\tau$ is $\Delta^0_1$-definable. The obvious restriction~$\mu\!\restriction\!\lo_0$ will then exist as a set, and we have the following.

\begin{lemma}\label{def:morphism-dilators}
Assume that $\mu:D\To E$ is a morphism of class-sized dilators. Then the restriction $\mu\!\restriction\!\lo_0:D\!\restriction\!\lo_0\To E\!\restriction\!\lo_0$ is a morphism of coded dilators.
\end{lemma}

In order to prove a converse, we will need the following fact, which applies in the coded as well as in the class-sized case. The result is due to Girard~\cite[Proposition~2.3.15]{girard-pi2}; a proof that uses our terminology can be found in~\cite[Lemma~2.19]{freund-rathjen_derivatives}.

\begin{lemma}\label{lem:transfos-respect-supp}
We have $\supp^E\circ\mu=\supp^D$ for any morphism $\mu:D\To E$ between predilators $D=(D,\supp^D)$ and $E=(E,\supp^E)$.
\end{lemma}

The lemma ensures that $(n,\sigma)\in\tr(D)$ implies $(n,\mu_n(\sigma))\in\tr(E)$, which is needed in order to justify the following construction.

\begin{definition}\label{def:morphs-dils-extend}
Consider a morphism $\mu:D\To E$ of coded predilators. For each order~$X$ we define $\overline\mu_X:\overline D(X)\to\overline E(X)$ by setting $\overline\mu_X(a,\sigma)=(a,\mu_{|a|}(\sigma))$.
\end{definition}

The following has been verified in~\cite[Lemma~2.21]{freund-rathjen_derivatives}.

\begin{lemma}\label{lem:morphism-dils-extend}
If $\mu:D\To E$ is a morphism of coded predilators, then $\overline\mu:\overline D\To\overline E$ is a morphism of class-sized predilators.
\end{lemma}

It is straightforward to see that $(\cdot)\!\restriction\!\lo_0$ is a functor from the category of class-sized predilators to the category of coded predilators, and that $\overline{(\cdot)}$ is a functor in the converse direction. Together with Theorem~\ref{thm:class-coded-equiv}, the following shows that $\eta$ is a natural isomorphism between the composition $\overline{(\cdot)\!\restriction\!\lo_0}$ and the identity on the category of class-sized predilators.

\begin{proposition}\label{prop:eta-natural-in-D}
We have $\eta^E\circ\overline{\mu\!\restriction\!\lo_0}=\mu\circ\eta^D$ whenever $\mu:D\To E$ is a morphism of class-sized predilators.
\end{proposition}
\begin{proof}
Using the naturality of $\mu$, we get
\begin{multline*}
(\eta^E\circ\overline{\mu\!\restriction\!\lo_0})_X(a,\sigma)=\eta^E_X(a,\mu_{|a|}(\sigma))=E(\iota_a^X\circ\en_a)\circ\mu_{|a|}(\sigma)=\\
=\mu_X\circ D(\iota_a^X\circ\en_a)(\sigma)=(\mu\circ\eta^D)_X(\sigma)
\end{multline*}
for each order~$X$ and each element $(a,\sigma)\in\overline{D\!\restriction\!\lo_0}(X)$.
\end{proof}

One can also start with a coded predilator~$D$, form the class-sized extension~$\overline D$, and then revert to the coded restriction $\overline D\!\restriction\!\lo_0$. By mapping $(a,\sigma)\in\overline D(n)$ to the element $D(\iota_a^X\circ\en_a)(\sigma)\in D(n)$, we get a natural isomorphism~$\overline D\!\restriction\!\lo_0\cong D$, as verified in~\cite[Lemma~2.6]{freund-rathjen_derivatives}. Analogous to the proof of Proposition~\ref{prop:eta-natural-in-D}, one can show that the construction is natural in~$D$. Together, these considerations show that the category of class-sized predilators is equivalent to the category of coded predilators (in the sense of~\cite[Section~IV.4]{maclane-working}).

In the following, we show that the trace of a predilator plays an analogous role to the underlying set of a linear order. This will yield characterizations of direct limits and pullbacks in the category of predilators. As in the first part of this section, the results are due to Girard~\cite{girard-pi2}, but our formalism is quite different. Let us first specify what we mean by the range of a morphism of predilators. The following construction is justified by Lemma~\ref{lem:transfos-respect-supp}.

\begin{definition}\label{def:trace-morphism}
Given a morphism $\mu:D\To E$ of predilators, we define an injective function $\tr(\mu):\tr(D)\to\tr(E)$ by setting
\begin{equation*}
\tr(\mu)(n,\sigma)=(n,\mu_n(\sigma)).
\end{equation*}
The range of~$\mu$ is defined as the set $\rng(\mu):=\rng(\tr(\mu))\subseteq\tr(E)$.
\end{definition}

Girard~\cite[Theorem~4.2.5]{girard-pi2} has shown that any subset $A\subseteq\tr(D)$ gives rise to a predilator~$D[A]$ and a morphism $\iota[A]:D[A]\To D$ with $\rng(\iota[A])=A$. In the following we recover this result in our terminology. As preparation, we consider an element $\sigma\nf D(\iota_a^X\circ\en_a)(\sigma_0)\in D(X)$ and an order embedding $f:X\to Y$. For $b:=[f]^{<\omega}(a)\in[Y]^{<\omega}$ we have $|b|=|a|$ and $f\circ\iota_a^X\en_a=\iota_b^Y\circ\en_b$, as both functions enumerate the finite order~$b$. Hence $D(f)(\sigma)\nf D(\iota_b^Y\circ\en_b)(\sigma_0)\in D(Y)$ depends on the same trace element $(|b|,\sigma_0)=(|a|,\sigma_0)\in\tr(D)$. In the context of the following construction, this justifies the definition of $D[A](f)$.

\begin{definition}\label{def:D[A]}
Consider a predilator~$D=(D,\supp^D)$ and a set $A\subseteq\tr(D)$. For each order~$X$  (with $X\in\lo$ in the class-sized and $X\in\lo_0$ in the coded case) we define a suborder $D[A](X)\subseteq D(X)$ by stipulating
\begin{equation*}
\sigma\in D[A](X)\quad:\Leftrightarrow\quad (|a|,\sigma_0)\in A\text{ for }\sigma\nf D(\iota_a^X\circ\en_a)(\sigma_0).
\end{equation*}
For a morphism $f:X\to Y$, let $D[A](f):D[A](X)\to D[A](Y)$ be the restriction of the embedding~$D(f):D(X)\to D(Y)$. We also define $\supp^{D[A]}_X:D[A](X)\to[X]^{<\omega}$ as the restriction of the support function $\supp^D_X:D(X)\to[X]^{<\omega}$.
\end{definition}

Note that the relation $\sigma\in D[A](X)$ is $\Delta^0_1$-definable with set parameter~$A$, by the proof of Theorem~\ref{thm:class-coded-equiv} and the discussion that precedes it.

\begin{lemma}\label{lem:D[A]-predil}
If $D$ is a (pre-)dilator, then so is $D[A]$, for any $A\subseteq\tr(D)$.
\end{lemma}
\begin{proof}
We only verify the support condition from part~(ii) of Definition~\ref{def:dilator}, since all other properties are immediate. Consider an embedding $f:X\to Y$ and an element $\tau\in D[A](Y)$ with $\rng(f)\supseteq\supp^{D[A]}_Y(\tau)=\supp^D_Y(\tau)$. The support condition for~$D$ yields $\tau=D(f)(\sigma)$ for some $\sigma\in D(X)$. Write $\sigma\nf D(\iota_a^X\circ\en_a)(\sigma_0)$, and argue as above to get $\tau\nf D(\iota_b^Y\circ\en_b)(\sigma_0)$ with $b=[f]^{<\omega}$. In view of $\tau\in D[A](Y)$ we can conclude $(|a|,\sigma_0)=(|b|,\sigma_0)\in A$ and then $\sigma\in D[A](X)$. It follows that we have $\tau=D(f)(\sigma)=D[A](f)(\sigma)\in\rng(D[A](f))$, as needed.
\end{proof}

One readily checks that the following yields a morphism of predilators.

\begin{definition}\label{def:iota[A]}
Consider a predilator~$D$ and a set $A\subseteq\tr(D)$. To define a morphism $\iota[A]:D[A]\To D$, we declare that each component \mbox{$\iota[A]_X:D[A](X)\hookrightarrow D(X)$} is the inclusion map.
\end{definition}

Let us verify the promised property:

\begin{lemma}\label{lem:iota[A]}
We have $\rng(\iota[A])=A$ for each predilator~$D$ and each $A\subseteq\tr(D)$.
\end{lemma}
\begin{proof}
To establish the first inclusion, we consider $(n,\sigma)\in\rng(\iota[A])\subseteq\tr(D)$. Since $\iota[A]_n:D[A](n)\hookrightarrow D(n)$ is the inclusion, we must have $(n,\sigma)\in\tr(D[A])$, which entails $\sigma\in D[A](n)$. Due to $(n,\sigma)\in\tr(D)$ we have $\sigma\nf D(\iota_n^n\circ\en_n)(\sigma)$, where both $\iota_n^n$ and $\en_n$ is the identity on $n=\{0,\dots,n-1\}$. Now $(n,\sigma)\in A$ follows by the equivalence that defines $D[A](n)$. For the converse inclusion, we consider an arbitrary element $(n,\sigma)\in A\subseteq\tr(D)$. Once again we have $\sigma\nf D(\iota_n^n\circ\en_n)(\sigma)$, so that we get $\sigma\in D[A](n)$. Together with $\supp^{D[A]}_n(\sigma)=\supp^D_n(\sigma)=n$ we obtain $(n,\sigma)\in\tr(D[A])$ and then $(n,\sigma)=(n,\iota[A]_n(\sigma))=\tr(\iota[A])(n,\sigma)\in\rng(\iota[A])$.
\end{proof}

The following result entails that $D[A]$ and $\iota[A]$ are essentially unique.

\begin{proposition}\label{prop:trace-inclusion-to-morphism}
For all morphisms $\mu:D\To E$ and $\mu':D'\To E$ of predilators, the following are equivalent:
\begin{enumerate}[label=(\roman*)]
\item we have $\rng(\mu)\subseteq\rng(\mu')$,
\item there is a (necessarily unique) morphism $\nu:D\To D'$ with $\mu'\circ\nu=\mu$.
\end{enumerate}
\end{proposition}
\begin{proof}
To show that (ii) implies (i), we consider an arbitrary element $\tr(\mu)(n,\sigma)$ of $\rng(\mu)$. Assuming $\mu'\circ\nu=\mu$, we obtain
\begin{equation*}
\tr(\mu)(n,\sigma)=(n,\mu_n(\sigma))=(n,\mu'_n\circ\nu_n(\sigma))=\tr(\mu')(n,\nu_n(\sigma))\in\rng(\mu').
\end{equation*}
We now show that (i) implies~(ii). Given $\sigma\nf D(\iota_a^X\circ\en_a)(\sigma_0)\in D(X)$, we note
\begin{equation*}
(|a|,\mu_{|a|}(\sigma_0))=\tr(\mu)(|a|,\sigma_0)\in\rng(\mu).
\end{equation*}
Assuming $\rng(\mu)\subseteq\rng(\mu')$, there is a unique $\tau_0\in D'(|a|)$ with $\mu_{|a|}(\sigma_0)=\mu'_{|a|}(\tau_0)$. We can then define a function $\nu_X:D(X)\to D'(X)$ by setting
\begin{equation*}
\nu_X(\sigma)=D'(\iota_a^X\circ\en_a)(\tau_0)\quad\text{for $\sigma\nf D(\iota_a^X\circ\en_a)(\sigma_0)$ and $\mu_{|a|}(\sigma_0)=\mu'_{|a|}(\tau_0)$}.
\end{equation*}
For $\sigma\in D(X)$ as specified, one readily computes
\begin{multline*}
\mu_X'\circ\nu_X(\sigma)=\mu_X'\circ D'(\iota_a^X\circ\en_a)(\tau_0)=E(\iota_a^X\circ\en_a)\circ\mu'_{|a|}(\tau_0)=\\
=E(\iota_a^X\circ\en_a)\circ\mu_{|a|}(\sigma_0)=\mu_X\circ D(\iota_a^X\circ\en_a)(\sigma_0)=\mu_X(\sigma).
\end{multline*}
It follows that $\nu_X$ is an embedding, i.\,e., a morphism in~$\lo$. To establish naturality, consider an embedding $f:X\to Y$ and an element $\sigma\nf D(\iota_a^X\circ\en_a)(\sigma_0)\in D(X)$. For $b:=[f]^{<\omega}(a)$ we get $f\circ\iota_a^X\circ\en_a=\iota_b^Y\circ\en_b$ and $D(f)(\sigma)\nf D(\iota_b^Y\circ\en_b)(\sigma_0)$ as before. Consider the unique element $\tau_0\in D'(|a|)$ with $\mu_{|a|}(\sigma_0)=\mu'_{|a|}(\tau_0)$. In view of $|a|=|b|$ we get $\mu_{|b|}(\sigma_0)=\mu'_{|b|}(\tau_0)$ and hence
\begin{equation*}
\nu_Y\circ D(f)(\sigma)=D'(\iota_b^Y\circ\en_b)(\tau_0)=D'(f)\circ D'(\iota_a^X\circ\en_a)(\tau_0)=D'(f)\circ\nu_X(\sigma).
\end{equation*}
To see that $\nu$ is unique, it suffices to observe that each component $\mu'_X$ is an embedding and hence injective.
\end{proof}

Based on the previous considerations, it is straightforward to establish the following result, which is due to Girard~\cite[Theorem~4.2.7]{girard-pi2}.

\begin{proposition}\label{prop:pullback-dilator}
In the category of predilators, any two morphisms $\mu^i:E_i\To E$ with $i=0,1$ have a pullback. Two morphisms $\nu^i:D\To E_i$ with $\mu^0\circ\nu^0=\mu^1\circ\nu^1$ form a pullback of $\mu^0$ and $\mu^1$ if, and only if, we have $\rng(\mu^0)\cap\rng(\mu^1)\subseteq\rng(\mu^0\circ\nu^0)$.
\end{proposition}
Let us point out that the converse inclusion
\begin{equation*}
\rng(\mu^0\circ\nu^0)=\rng(\mu^1\circ\nu^1)\subseteq\rng(\mu^0)\cap\rng(\mu^1)
\end{equation*}
is automatic when we have $\mu^0\circ\nu^0=\mu^1\circ\nu^1$.
\begin{proof}
For existence we set $A:=\rng(\mu^0)\cap\rng(\mu^1)$ and consider $\iota[A]:E[A]\To E$. In view of $\rng(\iota[A])=A\subseteq\rng(\mu^i)$, we get morphisms $\xi^i:E[A]\To E_i$ with
\begin{equation*}
\mu^0\circ\xi^0=\iota[A]=\mu^1\circ\xi^1.
\end{equation*}
To show that these morphisms satisfy the universal property of pullbacks, we consider morphisms $\nu^i:D\To E_i$ with $\mu^0\circ\nu^0=\mu^1\circ\nu^1$. In view of
\begin{equation*}
\rng(\mu^0\circ\nu^0)=\rng(\mu^1\circ\nu^1)\subseteq\rng(\mu^0)\cap\rng(\mu^1)=A=\rng(\iota[A]),
\end{equation*}
there is a unique morphism $\zeta:D\To E[A]$ with
\begin{equation*}
\mu^0\circ\xi^0\circ\zeta=\mu^1\circ\xi^1\circ\zeta=\iota[A]\circ\zeta=\mu^0\circ\nu^0=\mu^1\circ\nu^1.
\end{equation*}
Since the components of $\mu^i$ are injective, we get $\xi^i\circ\zeta=\nu^i$, and $\zeta$ is still unique with this property. In order to establish the characterization in the second part of the proposition, we assume that the morphisms $\nu^i:D\To E_i$ form a pullback. Due to the universal property, we obtain a morphism $\chi:E[A]\To D$ with $\nu^i\circ\chi=\xi^i$, for $\xi^i:E[A]\To E_i$ as above. We can deduce the required inclusion
\begin{equation*}
\rng(\mu^0)\cap\rng(\mu^1)=\rng(\iota[A])=\rng(\mu^0\circ\xi^0)\subseteq \rng(\mu^0\circ\nu^0\circ\chi)\subseteq\rng(\mu^0\circ\nu^0).
\end{equation*}
For the converse implication, consider morphisms $\nu^i:D\To E_i$ with $\mu^0\circ\nu^0=\mu^1\circ\nu^1$ and $A=\rng(\mu^0)\cap\rng(\mu^1)\subseteq\rng(\mu^0\circ\nu^0)$. Since the converse inclusion is automatic, we now get an isomorphism $\chi:E[A]\To D$ with \mbox{$\iota[A]=\mu^0\circ\nu^0\circ\chi=\mu^1\circ\nu^1\circ\chi$} (as Proposition~\ref{prop:trace-inclusion-to-morphism} provides morphisms in both directions, which are inverses by uniqueness). Using $\chi$, the universal property for the morphisms $\nu^i:D\To E_i$ is readily reduced to the one for the morphisms $\xi^i:E[A]\To E_i$ from above.
\end{proof}

Girard~\cite[Theorem~4.4.4]{girard-pi2} has shown that any direct system in the category of predilators has a direct limit. However, the direct limit of a system of dilators does not need to be a dilator itself (i.\,e., it may not preserve well foundedness). Direct limits of predilators are constructed pointwise: for each argument, one forms the corresponding limit in the category of linear orders. We will see in Section~\ref{sect:construct-fixed-points} that (a~particular instance of) this construction can be carried out in~$\aca_0$. In the next section we will define ptykes in terms of a support condition (analogous to part~(ii) of Definition~\ref{def:dilator}), which is motivated by the following characterization.

\begin{proposition}\label{prop:direct-limit-dilators}
In the category of predilators, consider a direct system of objects $D_i$ and morphisms $\mu^{ij}:D_i\To D_j$, indexed by a directed set~$I$. For a collection of morphisms $\nu^i:D_i\To D$ with $\nu^j\circ\mu^{ij}=\nu^i$, the following are equivalent:
\begin{enumerate}[label=(\roman*)]
\item the morphisms $\nu^i:D_i\To D$ form a direct limit of the given system,
\item we have $\tr(D)=\bigcup\{\rng(\nu^i)\,|\,i\in I\}$.
\end{enumerate}
\end{proposition}
\begin{proof}
To show that~(i) implies~(ii), we set $A:=\bigcup\{\rng(\nu^i)\,|\,i\in I\}$ and consider the morphism $\iota[A]:D[A]\To D$. For each $i\in I$, the inclusion $\rng(\nu^i)\subseteq\rng(D[A])$ yields a unique $\xi^i:D_i\To D[A]$ with $\iota[A]\circ\xi^i=\nu^i$. Due to $\iota[A]\circ\xi^j\circ\mu^{ij}=\nu^j\circ\mu^{ij}=\nu^i$, uniqueness entails $\xi^j\circ\mu^{ij}=\xi^i$. Assuming~(i), the universal property of direct limits provides a morphism~$\chi:D\To D[A]$ with $\chi\circ\nu^i=\xi^i$ for each~$i\in I$. In view of
\begin{equation*}
\iota[A]\circ\chi\circ\nu=\iota[A]\circ\xi^i=\nu^i,
\end{equation*}
the uniqueness part of the universal property entails that $\iota[A]\circ\chi$ is the identity on the predilator~$D$. We can deduce the crucial inclusion
\begin{equation*}
\tr(D)=\rng(\iota[A]\circ\chi)\subseteq\rng(\iota[A])=A=\bigcup\{\rng(\nu^i)\,|\,i\in I\}.
\end{equation*}
In order to show that~(ii) implies~(i), we consider a family of morphisms~$\xi^i:D_i\To E$ with $\xi^j\circ\mu^{ij}=\xi^i$. Given an order~$X$, we recall that any $\sigma\in D(X)$ has a unique normal form $\sigma\nf D(\iota_a^X\circ\en_a)(\sigma_0)$ with $(|a|,\sigma_0)\in\tr(D)$. Assuming~(ii), we get an~$i\in I$ with $(|a|,\sigma_0)\in\rng(\nu^i)$, which amounts to $\sigma_0\in\rng(\nu^i_{|a|})$. In order to obtain a morphism $\zeta:D\To E$ with $\zeta\circ\nu^i=\xi^i$, we can only set
\begin{equation*}
\zeta_X(\sigma)=E(\iota_a^X\circ\en_a)\circ\xi^i_{|a|}(\tau)\quad\text{for $\sigma\nf D(\iota_a^X\circ\en_a)(\sigma_0)$ and $\sigma_0=\nu^i_{|a|}(\tau)$}.
\end{equation*}
To see that this is well defined, we consider an equality $\nu^i_{|a|}(\tau)=\nu^j_{|a|}(\tau')$. Since~$I$ is directed, we may pick a $k$ that lies above both~$i$ and~$j$. We then get
\begin{equation*}
\nu^k_{|a|}\circ\mu^{ik}_{|a|}(\tau)=\nu^i_{|a|}(\tau)=\nu^j_{|a|}(\tau')=\nu^k_{|a|}\circ\mu^{jk}_{|a|}(\tau'),
\end{equation*}
which implies $\mu^{ik}_{|a|}(\tau)=\mu^{jk}_{|a|}(\tau')$ and then $\xi^i_{|a|}(\tau)=\xi^j_{|a|}(\tau')$, as needed to see that $\zeta_X$ is well defined. Given $\sigma\nf D_i(\iota_a^X\circ\en_a)(\sigma_0)\in D_i(X)$, one can invoke Lemma~\ref{lem:transfos-respect-supp} to see that $\nu^i_X(\sigma)=D(\iota_a^X\circ\en_a)(\nu^i_{|a|}(\sigma_0))$ is in normal form as well. This yields
\begin{equation*}
\zeta_X\circ\nu^i_X(\sigma)=E(\iota_a^X\circ\en_a)\circ\xi^i_{|a|}(\sigma_0)=\xi^i_X\circ D_i(\iota_a^X\circ\en_a)(\sigma_0)=\xi^i_X(\sigma).
\end{equation*}
In order to deduce that $\zeta_X$ is an order embedding, it suffices to note that any $\sigma\nf D(\iota_a^X\circ\en_a)(\sigma_0)\in D(X)$ with $(|a|,\sigma_0)\in\rng(\nu^i)$ lies in the range of~$\nu^i_X$. Finally, we verify that $\zeta$ is natural. Consider an embedding~$f:X\to Y$ and an element $\sigma\nf D(\iota_a^X\circ\en_a)(\sigma_0)\in D(X)$. As before, we see that $b:=[f]^{<\omega}(a)$ leads to $|b|=|a|$ and $f\circ\iota_a^X\circ\en_a=\iota_b^Y\circ\en_b$, which entails $D(f)(\sigma)\nf D(\iota_b^Y\circ\en_b)(\sigma_0)$. For $\sigma_0=\nu^i_{|a|}(\tau)=\nu^i_{|b|}(\tau)$ we get
\begin{equation*}
\zeta_Y\circ D(f)(\sigma)=E(\iota_b^Y\circ\en_b)\circ\xi^i_{|b|}(\tau)=E(f\circ\iota_a^X\circ\en_a)\circ\xi^i_{|a|}(\tau)=E(f)\circ\zeta_X(\sigma),
\end{equation*}
as needed for naturality.
\end{proof}

\section{Ptykes in second order arithmetic}\label{sect:ptykes}

If we consider (well founded) linear orders as objects of ground type, then \mbox{(pre-)}\allowbreak{}dilators form the first level in a hierarchy of countinuous functionals of finite type. The functionals in this hierarchy are called (pre-)ptykes (singular ptyx; see~\cite{girard-normann92}). On the second level of this hierarchy, we have transformations that take predilators as input. The output can be a linear order, a predilator, or even another functional from the second level---at least intuitively, these possibilities are equivalent modulo Currying. We will only consider transformations of predilators into predilators, and the term $2$-ptyx will be reserved for these. The number~$2$ indicates the type level and will often be omitted. In the present section, we define $2$-ptykes in terms of a support condition, which is analogous to part~(ii) of Definition~\ref{def:dilator}; we show that the given definition is equivalent to the original definition by Girard, which invokes direct limits and pullbacks; and we discuss an example.

In the previous section we have discussed the category of predilators. We have seen that the trace~$\tr(D)$ of a predilator~$D$ plays an analogous role to the underlying set of a linear order. Given a morphism $\nu:D\To E$ of predilators, we have constructed a function $\tr(\nu):\tr(D)\to\tr(E)$ with range~\mbox{$\rng(\nu):=\rng(\tr(\nu))\subseteq\tr(E)$}. One readily checks that this yields a functor $\tr(\cdot)$ from the category of predilators to the category of sets, as presupposed in part~(ii) of the following definition. The formalization of the definition in second order arithmetic will be discussed below.

\begin{definition}\label{def:ptyx}
A $2$-preptyx consists of
\begin{enumerate}[label=(\roman*)]
\item a functor $P$ from predilators to predilators and
\item a natural transformation $\Supp:\tr(P(\cdot))\to[\tr(\cdot)]^{<\omega}$ such that we have
\begin{equation*}
\{\sigma\in\tr(P(E))\,|\,\Supp_E(\sigma)\subseteq\rng(\mu)\}\subseteq\rng(P(\mu))
\end{equation*}
for any morphism $\mu:D\To E$ of predilators.
\end{enumerate}
If $P(D)$ is a dilator for every dilator~$D$, then $P=(P,\Supp)$ is called a $2$-ptyx.
\end{definition}

As in the case of dilators, we will refer to the inclusion in part~(ii) as the support condition. The converse inclusion is, once again, automatic: given an arbitrary element $\sigma=\tr(P(\mu))(\sigma_0)\in\rng(P(\mu))$, we can invoke naturality to get
\begin{equation*}
\Supp_E(\sigma)=\Supp_E(\tr(P(\mu))(\sigma_0))=[\tr(\mu)]^{<\omega}(\Supp_D(\sigma_0))\subseteq\rng(\mu).
\end{equation*}
Clause~(i) of Definition~\ref{def:ptyx} is not completely precise, because we have distinguished between coded and class-sized predilators---even though the two notions give rise to equivalent categories (cf.~the discussion after Proposition~\ref{prop:eta-natural-in-D}). We can and will consider preptykes as classes-sized functions. However, the arguments of these functions should certainly be set-sized. This means that the arguments of a preptyx must be coded predilators in the sense of Definition~\ref{def:coded-predil}. We agree that the values of our preptykes are coded predilators as well, even though this is less essential. To define $P(D)$ as a coded predilator, it suffices to specify its values on finite orders $n=\{0,\dots,n-1\}\in\lo_0$ and morphisms between them. In practice, we often define the action of $P(D)$ on (morphisms of) infinite linear orders as well; this means that we describe a class-sized predilator of which~$P(D)$ is the coded restriction.

Whenever we speak of a preptyx, we assume that it is given by $\Delta^0_1$-definitions of the following relations, possibly with additional number and set parameters: First, we have the relations
\begin{equation*}
\sigma\in P(D)(X),\quad\sigma<_{P(D)(X)}\tau,\quad P(D)(f)(\sigma)=\tau,\quad \supp^{P(D)}_X(\sigma)=a
\end{equation*}
that define the predilator~$P(D)=(P(D),\supp^{P(D)})$ relative to the coded predilator~$D\subseteq\mathbb N$ (cf.~the discussion after Definition~\ref{def:coded-predil}). As in the previous section, these $\Delta^0_1$-definitions ensure that the coded predilator $P(D)$ exists as a set, already over~$\rca_0$. Secondly, we require a $\Delta^0_1$-formula $P(\mu)_X(\sigma)=\tau$ that defines the morphism $P(\mu)$ relative to the morphism $\mu\subseteq\mathbb N$ (cf.~the discussion before Definition~\ref{def:morphism-dilators}). Finally, we demand a $\Delta^0_1$-definition of the relation~$\Supp_D(\sigma)=a$, where $a$ refers to the numerical code of a finite subset of~$\tr(D)$.

By using parameters, one can quantify over $\Delta^0_1$-definable families of preptykes. In the following, all general definitions and results should be read as schemas, with one instance for each $\Delta^0_1$-definable family. We will not construct a universal family of $2$-preptykes in detail, as this is quite technical and not strictly necessary for the present paper. Let us, nevertheless, sketch the construction: Given a predilator~$D$ and a finite set $a\in[\tr(D)]^{<\omega}$, the constructions from the previous section yield a predilator~$D[a]$ and a morphism $\iota[a]:D[a]\To D$ with~$\rng(\iota[a])=a$. For $a\subseteq b$, Proposition~\ref{prop:trace-inclusion-to-morphism} ensures that there is a unique morphism $\nu^{ab}:D[a]\To D[b]$ with $\iota[b]\circ\nu^{ab}=\iota[a]$. This turns the collection of predilators~$D[a]$ into a directed system indexed by~$[\tr(D)]^{<\omega}$. Invoking Proposition~\ref{prop:direct-limit-dilators}, we learn that the morphisms $\iota[a]:D[a]\To D$ form a direct limit. Given a preptyx~$P$, Proposition~\ref{prop:ptykes-direct-limit} below entails that $P(D)$ is the (essentially unique) limit of the system of predilators $P(D[a])$ and morphisms $P(\nu^{ab}):P(D[a])\To P(D[b])$. In other words, $P$~is essentially determined by its action on (morphisms between) predilators with finite trace. The latter are determined by a finite amount of information (by~\cite[Proposition~4.3.7]{girard-pi2} or, implicitly, Definition~\ref{def:extend-coded-dil} above). By restricting to arguments with finite trace, one can thus define a notion of coded preptyx, analogous to Definition~\ref{def:coded-predil}. The following result is the main ingredient of the construction that we have sketched. It also shows that our definition of $2$-ptykes coincides with Girard's original one.

\begin{proposition}\label{prop:ptykes-direct-limit}
The following are equivalent for an endofunctor~$P$ of predilators:
\begin{enumerate}[label=(\roman*)]
\item The functor~$P$ preserves direct limits and pullbacks.
\item There is a natural transformation $\Supp:\tr(P(\cdot))\to[\tr(\cdot)]^{<\omega}$ that satisfies the support condition from part~(ii) of Definition~\ref{def:ptyx}.
\end{enumerate}
If a natural transformation as in~(ii) exists, then it is unique.
\end{proposition}
\begin{proof}
We first assume~(ii) and deduce~(i). To show that~$P$ preserves direct limits, we consider a direct system of predilators~$D_i$ and morphisms $\mu^{ij}:D_i\To D_j$, indexed by a directed set~$I$. Assume that the morphisms $\nu^i:D_i\To D$ form a direct limit. We need to show that the morphisms $P(\nu^i):P(D_i)\To P(D)$ form a direct limit of the system of objects $P(D_i)$ and morphisms $P(\mu^{ij}):P(D_i)\To P(D_j)$. In view of Proposition~\ref{prop:direct-limit-dilators}, we have $\tr(D)=\bigcup\{\rng(\nu^i)\,|\,i\in I\}$ and need to deduce
\begin{equation*}
\tr(P(D))=\bigcup\{\rng(P(\nu^i))\,|\,i\in I\}.
\end{equation*}
Here the inclusion $\supseteq$ is trivial, since $\rng(P(\nu^i))$ is defined as the range of the function $\tr(P(\nu^i)):\tr(P(D_i))\to\tr(P(D))$. Given an arbitrary $\sigma\in\tr(P(D))$, we invoke~(ii) to consider the finite set $\Supp_D(\sigma)\subseteq\tr(D)$. The latter is contained in a single set~$\rng(\nu^i)$, since~$I$ is directed. By the support condition we can conclude \mbox{$\sigma\in\rng(P(\nu^i))$}, as required. To show that~$P$ preserves pullbacks one argues similarly, based on the characterization from Proposition~\ref{prop:pullback-dilator}. Next, we show that there is at most one natural transformation as in~(ii). Given such a transformation, the support condition ensures that $\Supp_D(\sigma)=:a=\rng(\iota[a])$ implies $\sigma\in\rng(P(\iota[a]))$, where $\iota[a]:D[a]\To D$ is the morphism from Lemma~\ref{lem:iota[A]}. Conversely, naturality entails that $\sigma\in\rng(P(\iota[a]))$ implies $\Supp_D(\sigma)\subseteq\rng(\iota[a])$, as above. Hence $\Supp_D(\sigma)$ is determined as the smallest set $a\in[\tr(D)]^{<\omega}$ with $\sigma\in\rng(P(\iota[a]))$. Assuming~(i), we now show that a smallest set with this property exists for any element~$\sigma\in\tr(P(D))$. Above, we have seen that the morphisms $\iota[a]:D[a]\To D$ form a limit of some directed system indexed by $a\in[\tr(D)]^{<\omega}$. Given that~$P$ preserves direct limits, the same holds for the morphisms $P(\iota[a]):P(D[a])\To P(D)$. Now Proposition~\ref{prop:direct-limit-dilators} yields an $a\in[\tr(D)]^{<\omega}$ with $\sigma\in\rng(P(\iota[a]))$. We may assume that~$a$ is minimal, in the sense that no proper subset has the same property. To conclude that~$a$ is smallest, we need to show that $\sigma\in\rng(P(\iota[b]))$ implies $a\subseteq b$. Invoking Proposition~\ref{prop:trace-inclusion-to-morphism}, define $\nu^a:D[a\cap b]\To D[a]$ and $\nu^b:D[a\cap b]\To D[b]$ by stipulating $\iota[a]\circ\nu^a=\iota[a\cap b]$ and $\iota[b]\circ\nu^b=\iota[a\cap b]$. From Lemma~\ref{lem:iota[A]} and Proposition~\ref{prop:pullback-dilator} we know that $\nu^a$ and $\nu^b$ form a pullback of $\iota[a]$ and $\iota[b]$. By~(i) and Proposition~\ref{prop:pullback-dilator} we can conclude
\begin{equation*}
\sigma\in\rng(P(\iota[a]))\cap\rng(P(\iota[b]))\subseteq\rng(P(\iota[a]\circ\nu^a))=\rng(P(\iota[a\cap b])).
\end{equation*}
The minimality of~$a$ yields $a\cap b=a$ and hence $a\subseteq b$, as required. To establish~(ii), we now define $\Supp_D(\sigma)$ as the smallest $a\in[\tr(D)]^{<\omega}$ with $\sigma\in\rng(P(\iota[a]))$. To establish the support condition, it suffices to note that any morphism $\mu:D_0\To D$ with $a\subseteq\tr(\mu)$ factors through $\iota[a]$, by Proposition~\ref{prop:trace-inclusion-to-morphism}. It remains to check that the given definition is natural. For this purpose we consider a morphism~\mbox{$\mu:D\To E$} and an arbitrary element $\sigma\in\tr(P(D))$. Writing $a:=\Supp_D(\sigma)$, we need to show that $\Supp_E(\tr(P(\mu))(\sigma))$ is equal to $b:=[\tr(\mu)]^{<\omega}(a)=\rng(\mu\circ\iota[a])$. Consider the morphism~$\iota[b]:E[b]\To E$, and invoke Proposition~\ref{prop:trace-inclusion-to-morphism} to obtain $\mu^a:D[a]\To E[b]$ with $\iota[b]\circ\mu^a=\mu\circ\iota[a]$. By the choice of~$a$, write $\sigma=\tr(P(\iota[a]))(\sigma_0)$ and compute
\begin{equation*}
\tr(P(\mu))(\sigma)=\tr(P(\mu\circ\iota[a]))(\sigma_0)=\tr(P(\iota[b]\circ\mu^a))(\sigma_0)\in\rng(P(\iota[b])).
\end{equation*}
This yields $d:=\Supp_E(\tr(P(\mu))(\sigma))\subseteq b$, which allows us to consider the set $c\subseteq a$ with $d=[\tr(\mu)]^{<\omega}(c)=\rng(\mu\circ\iota[c])$. Again by Proposition~\ref{prop:trace-inclusion-to-morphism}, we obtain a morphism $\nu:E[d]\To D[c]$ with $\mu\circ\iota[c]\circ\nu=\iota[d]$. The choice of~$d$ allows us to write
\begin{equation*}
\tr(P(\mu))(\sigma)=\tr(P(\iota[d]))(\sigma_1)=\tr(P(\mu\circ\iota[c]\circ\nu))(\sigma_1).
\end{equation*}
Since $\tr(P(\mu))$ is injective, we get $\sigma=\tr(P(\iota[c]\circ\nu))(\sigma_1)\in\rng(P(\iota[c]))$. By the minimality of $a=\Supp_D(\sigma)$, this yields $c=a$ and hence $d=[\tr(\mu)]^{<\omega}(c)=b$.
\end{proof}

The following variant of the support functions is convenient, because it does not force us to consider the trace of~$P(D)$, which can be hard to describe. We recall that every element $\sigma\in P(D)(X)$ has a unique normal form $\sigma\nf P(D)(\iota_a^X\circ\en_a)(\sigma_0)$ with $(|a|,\sigma_0)\in\tr(P(D))$, as shown in the proof of Theorem~\ref{thm:class-coded-equiv}.

\begin{definition}\label{def:ptyx-support-variant}
Consider a $2$-preptyx $P=(P,\Supp)$. For each predilator~$D$ and each linear order~$X$, we define a function
\begin{equation*}
\Supp_{D,X}:P(D)(X)\to[\tr(D)]^{<\omega}
\end{equation*}
by setting $\Supp_{D,X}(\sigma):=\Supp_D(|a|,\sigma_0)$ for $\sigma\nf P(D)(\iota_a^X\circ\en_a)(\sigma_0)$.
\end{definition}

Let us record basic properties of our modified support functions:

\begin{lemma}\label{lem:ptyx-support-variant}
For any preptyx~$P$, the functions $\Supp_{D,X}$ are natural in~$D$ and~$X$, in the sense that we have
\begin{equation*}
\Supp_{E,X}\circ P(\mu)_X=[\tr(\mu)]^{<\omega}\circ\Supp_{D,X}\quad\text{and}\quad\Supp_{D,Y}\circ P(D)(f)=\Supp_{D,X},
\end{equation*}
for any morphism $\mu:D\To E$ of predilators and any order embedding~$f:X\to Y$. Furthermore, the support condition
\begin{equation*}
\{\sigma\in P(E)(X)\,|\,\Supp_{E,X}(\sigma)\subseteq\rng(\mu)\}\subseteq\rng(P(\mu)_X)
\end{equation*}
is satisfied for any morphism $\mu:D\To E$ and any linear order~$X$.
\end{lemma}
\begin{proof}
To establish the naturality properties, we consider an element
\begin{equation*}
\sigma\nf P(D)(\iota_a^X\circ\en_a)(\sigma_0)\in P(D)(X).
\end{equation*}
If $\mu:D\To E$ is a morphism of predilators, then so is $P(\mu)$. Invoking Lemma~\ref{lem:transfos-respect-supp}, we see that $(|a|,\sigma_0)\in\tr(P(D))$ entails $(|a|,P(\mu)_{|a|}(\sigma_0))\in\tr(P(E))$. Hence
\begin{equation*}
P(\mu)_X(\sigma)=P(\mu)_X(P(D)(\iota_a^X\circ\en_a)(\sigma_0))=P(E)(\iota_a^X\circ\en_a)(P(\mu)_{|a|}(\sigma_0))
\end{equation*}
is in normal form. Using the naturality property from Definition~\ref{def:ptyx}, we get
\begin{multline*}
\Supp_{E,X}\circ P(\mu)_X(\sigma)=\Supp_E(|a|,P(\mu)_{|a|}(\sigma_0))=\Supp_E\circ\tr(P(\mu))(|a|,\sigma_0)=\\
=[\tr(\mu)]^{<\omega}\circ\Supp_D(|a|,\sigma_0)=[\tr(\mu)]^{<\omega}\circ\Supp_{D,X}(\sigma).
\end{multline*}
For naturality with respect to~$f:X\to Y$, it suffices to recall that we get
\begin{equation*}
P(D)(f)(\sigma)\nf P(D)(\iota_b^Y\circ\en_b)(\sigma_0)\quad\text{with}\quad b=[f]^{<\omega}(a),
\end{equation*}
as in the discussion that precedes Definition~\ref{def:D[A]} (note $|b|=|a|$). The support condition is readily deduced from the one in Definition~\ref{def:ptyx}.
\end{proof}

For the following converse, we recall that elements of~$\tr(P(D))$ have the form $(n,\sigma)$ with $\sigma\in P(D)(n)$, where $n=\{0,\dots,n-1\}$ is ordered as usual.

\begin{lemma}\label{lem:ptyx-support-variant-back}
Consider an endofunctor $P$ of predilators and a family of functions $\Supp_{D,X}$ with the properties from Lemma~\ref{lem:ptyx-support-variant}. We obtain a preptyx $(P,\Supp)$ if we define $\Supp_D:\tr(P(D))\to[\tr(D)]^{<\omega}$ by $\Supp_D(n,\sigma):=\Supp_{D,n}(\sigma)$.
\end{lemma}
\begin{proof}
Naturality in~$D$ is straightforward. It remains to check the support condition from part~(ii) of Definition~\ref{def:ptyx}. Consider an element $(n,\sigma)\in\tr(P(E))$ with
\begin{equation*}
\Supp_E(n,\sigma)=\Supp_{E,n}(\sigma)\subseteq\rng(\mu).
\end{equation*}
The support condition from Lemma~\ref{lem:ptyx-support-variant} yields $\sigma=P(\mu)_n(\sigma_0)$ with $\sigma_0\in P(D)(n)$. Crucially, Lemma~\ref{lem:transfos-respect-supp} ensures $(n,\sigma_0)\in\tr(P(D))$, so that we get
\begin{equation*}
(n,\sigma)=(n,P(\mu)_n(\sigma_0))=\tr(P(\mu))(n,\sigma_0)\in\rng(P(\mu)),
\end{equation*}
as required.
\end{proof}

The functions $\Supp_D:\tr(P(D))\to[\tr(D)]^{<\omega}$ are unique by Proposition~\ref{prop:ptykes-direct-limit}. As in the proof of the latter, the value $\Supp_{D,X}(\sigma)$ is uniquely determined as the smallest $a\subseteq\tr(D)$ with $\sigma\in\rng(P(\iota[a])_X)$. Due to uniqueness, the constructions from the previous two lemmas must me inverse to each other. The proof of Lemma~\ref{lem:ptyx-support-variant-back} does not use the assumption that the functions $\Supp_{D,X}$ are natural in~$X$. Hence the latter is automatic when the other properties from Lemma~\ref{lem:ptyx-support-variant}~are~satisfied.

The following example reveals a connection with a more familiar topic: it shows that the transformation of a normal function into its derivative is related to the notion of ptyx. A much simpler ptyx will be described in Example~\ref{ex:ptyx-succ} below.

\begin{example}\label{ex:deriv-ptyx}
A normal predilator consists of a predilator~$D=(D,\supp^D)$ and a natural family of functions $\mu^D_X:X\to D(X)$ such that we have
\begin{equation*}
\sigma<_{D(X)}\mu^D_X(x)\quad\Leftrightarrow\quad\supp^D_X(\sigma)\subseteq\{x'\in X\,|\,x'<_X x\},
\end{equation*}
for any linear order~$X$ and arbitrary elements $x\in X$ and $\sigma\in D(X)$. If $D=(D,\mu^D)$ is a normal dilator, then the induced function on the ordinals (which maps~$\alpha$ to the order type of~$D(\alpha)$) is normal in the usual sense, due to P.~Aczel~\cite{aczel-phd,aczel-normal-functors}. The latter has also shown that one can transform~$D$ into a normal (pre-)dilator~$\partial D$ that induces the derivative of the normal function induced by~$D$. Together with M.~Rathjen, the present author has established that the construction of~$\partial D$ can be implemented in~$\rca_0$. Over the latter, the statement that $\partial D$ is a dilator (i.\,e.,~preserves well foundedness) when the same holds for~$D$ is equivalent to \mbox{$\Pi^1_1$-bar} induction~\cite{freund-rathjen_derivatives,freund-ordinal-exponentiation}. The transformation of~$D$ into~$P(D):=\partial D$ is no~$2$-ptyx in the strict sense, since it only acts on normal predilators. Nevertheless, it is instructive to verify that all other conditions from Definition~\ref{def:ptyx} are satisfied. For this purpose, we recall that $\partial D(X)$ is recursively generated by the following clauses (cf.~\cite[Definition~4.1]{freund-rathjen_derivatives}):
\begin{itemize}
\item For each element $x\in X$, there is a term $\mu^{\partial D}_x\in\partial D(X)$.
\item Given a finite set $a\subseteq\partial D(X)$, we get a term $\xi\langle a,\sigma\rangle\in\partial D(X)$ for each $\sigma\in D(|a|)$ with $(|a|,\sigma)\in\tr(D)$, except when $\langle a,\sigma\rangle=\langle\{\mu^{\partial D}_x\},\mu^D_1(0)\rangle$ for some $x\in X$ (where $\mu^D_1:1=\{0\}\to D(1)$ witnesses the normality of~$D$).
\end{itemize}
To explain the exception in the second clause, we write $f$ for the normal function induced by~$D$ and $f'$ for its derivative. Intuitively speaking, a term $\xi\langle\{\mu^{\partial D}_x\},\mu^D_1(0)\rangle$ with $x=\alpha\in\beta=X$ would denote the ordinal~$f(f'(\alpha))$. The latter is equal to~$f'(\alpha)$, which is already represented by the term~$\mu^{\partial D}_x$. We refer to~\cite{freund-rathjen_derivatives} for the definitions that turn $X\mapsto\partial D(X)$ into a normal (pre-)dilator. To turn $D\mapsto\partial D=P(D)$ into a functor, we consider a morphism~$\nu:D\To E$ of normal predilators (which requires $\nu_1(\mu^D_1(0))=\mu^E_1(0)$, according to~\cite[Definition~2.20]{freund-rathjen_derivatives}). In order to obtain a morphism $P(\nu):P(D)\To P(E)$, define the components $P(\nu)_X:\partial D(X)\to\partial E(X)$ by the recursive clauses
\begin{align*}
P(\nu)_X(\mu^{\partial D}_x)&=\mu^{\partial E}_x,\\
P(\nu)_X(\xi\langle a,\sigma\rangle)&=\xi\langle[P(\nu)_X]^{<\omega}(a),\nu_{|a|}(\sigma)\rangle.
\end{align*}
By induction over the complexity of terms in~$\partial D(X)$, one can simultaneously verify that~$P(\nu)_X$ has values in~$\partial E(X)$ and preserves the order (cf.~\cite[Definition~4.3]{freund-rathjen_derivatives}). The simultaneous induction is needed to ensure that $[P(\nu)_X]^{<\omega}(a)$ and~$a$ have the same cardinality, which yields
\begin{equation*}
(|[P(\nu)_X]^{<\omega}(a)|,\nu_{|a|}(\sigma))=(|a|,\nu_{|a|}(\sigma))=\tr(\nu)(|a|,\sigma)\in\tr(E),
\end{equation*}
as required for~$P(\nu)_X(\xi\langle a,\sigma\rangle)\in\partial E(X)$. It is straightforward to verify that~$P(\nu)$ is natural (cf.~\cite[Definition~4.6]{freund-rathjen_derivatives}), and that its construction turns~$P$ into a functor. In order to show that we have something like a~$2$-ptyx (except for the restriction to normal predilators), we need to construct support functions
\begin{equation*}
\Supp_{D,X}:P(D)(X)=\partial D(X)\to[\tr(D)]^{<\omega}
\end{equation*}
as in Lemma~\ref{lem:ptyx-support-variant}. We recursively define
\begin{align*}
\Supp_{D,X}(\mu^{\partial D}_x)&=\emptyset,\\
\Supp_{D,X}(\xi\langle a,\sigma\rangle)&=\{(|a|,\sigma)\}\cup\bigcup\{\Supp_{D,X}(\rho)\,|\,\rho\in a\}.
\end{align*}
Naturality in~$D$ is readily verified. In order to establish the support condition, we consider a morphism $\nu:D\To E$ and show
\begin{equation*}
\Supp_{E,X}(\tau)\subseteq\rng(\nu)\quad\To\quad\tau\in\rng(P(\nu)_X)
\end{equation*}
by induction over $\tau\in P(E)(X)=\partial E(X)$. First observe that the conclusion holds for~$\tau=\mu^{\partial E}_x=P(\nu)_X(\mu^{\partial D}_x)$. Now consider $\tau=\xi\langle a,\sigma\rangle$ with $\Supp_{E,X}(\tau)\subseteq\rng(\nu)$. Inductively we get $a\subseteq\rng(P(\nu)_X)$, say $a=[P(\nu)_X]^{<\omega}(b)$ with $b\subseteq\partial D(X)$. We also have $(|a|,\sigma)\in\rng(\nu)$, which yields $\sigma=\nu_{|a|}(\sigma_0)=\nu_{|b|}(\sigma_0)$ for some $\sigma_0\in D(|b|)$. Invoking Lemma~\ref{lem:transfos-respect-supp}, we see that $(|a|,\sigma)\in\tr(E)$ entails $(|b|,\sigma_0)\in\tr(D)$. We can conclude $\xi\langle b,\sigma_0\rangle\in\partial D(X)$, since $\langle b,\sigma_0\rangle=\langle\{\mu^{\partial D}_x\},\mu^D_1(0)\rangle$ would entail $a=\{\mu^{\partial E}_x\}$ and $\sigma=\nu_1\circ\mu^D_1(0)=\mu^E_1(0)$. It follows that
\begin{equation*}
\tau=\xi\langle a,\sigma\rangle=\xi\langle[P(\nu)_X]^{<\omega}(b),\nu_{|b|}(\sigma_0)\rangle=P(\nu)_X(\xi\langle b,\sigma_0\rangle)
\end{equation*}
lies in the range of $P(\nu)_X$, as required.
\end{example}

\section{Normal $2$-ptykes}\label{sect:normality-ptykes}

The present section introduces a normality condition for $2$-ptykes, which is related to the notion of normal function on the ordinals. We then show that any $2$-ptyx is bounded by a normal one; this amounts to the construction of the ptyx~$P^*$ from the argument that was sketched in the introduction.

In Example~\ref{ex:deriv-ptyx}, we have recalled the notion of normal predilator. The crucial point is that any normal predilator~$D$ preserves initial segments (cf.~Girard's notion of flower~\cite{girard-pi2}): Assume the range of $f:X\to Y$ is an initial segment of~$Y$. To show that $\rng(D(f))\subseteq D(Y)$ is an initial segment as well, we consider an inequality $\sigma<_{D(Y)} D(f)(\tau)$. Aiming at a contradiction, we assume that $\sigma$ does not lie in the range of~$\rng(D(f))$. Invoking part~(ii) of Definition~\ref{def:dilator}, we may pick an element $y\in\supp^D_Y(\sigma)$ with $y\notin\rng(f)$. Since the latter is an initial segment, we get
\begin{equation*}
\supp^D_Y(D(f)(\tau))=[f]^{<\omega}(\supp^D_X(\tau))\subseteq\rng(f)\subseteq\{y'\in Y\,|\,y'<_Y y\}.
\end{equation*}
Also note $\supp^D_Y(\sigma)\not\subseteq\{y'\in Y\,|\,y'<_Y y\}$. If $\mu^D$ witnesses that~$D$ is normal, we get
\begin{equation*}
D(f)(\tau)<_{D(Y)}\mu^D_Y(y)\leq_{D(Y)}\sigma,
\end{equation*}
which contradicts our assumption. To explain why initial segments are relevant, we recall that direct limits of well orders do not need to be well founded. Indeed, any linear order is the direct limit of its finite (and hence well founded) suborders. However, the limit of a directed system of well orders is well founded when the range of any morphism in the system is an initial segment of its codomain. We now introduce a corresponding notion on the next type level:

\begin{definition}\label{def:segment}
A morphism $\nu:D\To E$ is called a segment if the range of each component $\nu_X:D(X)\to E(X)$ is an initial segment of the linear order~$E(X)$.
\end{definition}

We point out that the definition applies to both coded and class-sized predilators. In the coded case one only considers orders of the form $X=n=\{0,\dots,n-1\}$. The two variants of the definition are compatible with the equivalence between coded and class-sized predilators (cf.~Lemma~\ref{lem:morphism-dils-extend}):

\begin{lemma}\label{lem:ext-segment}
Consider a morphism $\nu:D\To E$ of coded predilators. If $\nu$ is a segment, then so is $\overline\nu:\overline D\To\overline E$.
\end{lemma}
\begin{proof}
Consider an order~$X$ and an inequality
\begin{equation*}
(a,\sigma)<_{\overline E(X)}\overline\nu_X(b,\tau)=(b,\nu_{|b|}(\tau)).
\end{equation*}
According to Definition~\ref{def:extend-coded-dil}, the latter amounts to
\begin{equation*}
E(|\iota_a^{a\cup b}|)(\sigma)<_{E(|a\cup b|)}E(|\iota_b^{a\cup b}|)\circ\nu_{|b|}(\tau)=\nu_{|a\cup b|}\circ D(|\iota_b^{a\cup b}|)(\tau)\in\rng(\nu_{|a\cup b|}).
\end{equation*}
Given that $\nu:D\To E$ is a segment, we get $E(|\iota_a^{a\cup b}|)(\sigma)=\nu_{|a\cup b|}(\sigma_0)$ for some element~$\sigma_0\in D(|a\cup b|)$. Invoking Lemma~\ref{lem:transfos-respect-supp}, we can compute
\begin{multline*}
\supp^D_{|a\cup b|}(\sigma_0)=\supp^E_{|a\cup b|}\circ\nu_{|a\cup b|}(\sigma_0)=\supp^E_{|a\cup b|}\circ E(|\iota_a^{a\cup b}|)(\sigma)=\\
=[|\iota_a^{a\cup b}|]^{<\omega}\circ\supp^E_{|a|}(\sigma)\subseteq\rng(|\iota_a^{a\cup b}|).
\end{multline*}
Now the support condition from part~(ii) of Definition~\ref{def:dilator} yields $\sigma_0=D(|\iota_a^{a\cup b}|)(\sigma_1)$ for some $\sigma_1\in D(|a|)$. We can deduce
\begin{equation*}
E(|\iota_a^{a\cup b}|)(\sigma)=\nu_{|a\cup b|}\circ D(|\iota_a^{a\cup b}|)(\sigma_1)=E(|\iota_a^{a\cup b}|)\circ\nu_{|a|}(\sigma_1),
\end{equation*}
which implies $\sigma=\nu_{|a|}(\sigma_1)$. Recall that $(a,\sigma)\in\overline D(X)$ entails $(|a|,\sigma)\in\tr(D)$. By Lemma~\ref{lem:transfos-respect-supp} we get $(|a|,\sigma_1)\in\tr(D)$ and hence $(a,\sigma_1)\in\overline D(X)$. This yields
\begin{equation*}
(a,\sigma)=(a,\nu_{|a|}(\sigma_1))=\overline\nu_X(a,\sigma_1)\in\rng(\overline\nu_X),
\end{equation*}
as needed to show that~$\overline\nu$ is a segment.
\end{proof}

The following is similar to the corresponding notion for dilators, which is itself related to the usual notion of normal function on the ordinals  (cf.~Example~\ref{ex:deriv-ptyx} and the discussion at the beginning of the present section).

\begin{definition}\label{def:ptyx-normal}
A $2$-preptyx~$P$ is called normal if $P(\nu):P(D)\To P(E)$ is a segment whenever the same holds for~$\nu:D\To E$.
\end{definition}

In Section~\ref{sect:construct-fixed-points} we will construct a minimal fixed point $D\cong P(D)$ of a given \mbox{$2$-preptyx}~$P$. We will see that~$D$ is a dilator (rather than just a predilator) when $P$ is a normal $2$-ptyx. The following example shows that normality is essential. We provide full details, because the construction will be needed later.

\begin{example}\label{ex:ptyx-succ}
Given a linear order~$X$, we define $X+1=X\cup\{\top\}$ as the order with a new biggest element~$\top$. If~$D$ is a (pre-)dilator, we get a (pre-)dilator~$D+1$ by setting $(D+1)(X):=D(X)+1$ and
\begin{align*}
(D+1)(f)(\sigma)&:=\begin{cases}
D(f)(\sigma) & \text{if $\sigma\in D(X)\subseteq (D+1)(X)$},\\
\top & \text{if $\sigma=\top$},
\end{cases}\\
\supp^{D+1}_X(\sigma)&:=\begin{cases}
\supp^D_X(\sigma) & \text{if $\sigma\in D(X)\subseteq (D+1)(X)$},\\
\emptyset & \text{if $\sigma=\top$},
\end{cases}
\end{align*}
where $f:X\to Y$ is an embedding. To turn $D\mapsto D+1$ into a functor, we transform each morphism $\nu:D\To E$ into the morphism $\nu+1:D+1\To E+1$ with
\begin{equation*}
(\nu+1)_X(\sigma):=\begin{cases}
\nu_X(\sigma) & \text{if $\sigma\in D(X)\subseteq (D+1)(X)$},\\
\top & \text{if $\sigma=\top$}.
\end{cases}
\end{equation*}
By Lemma~\ref{lem:ptyx-support-variant-back}, we obtain a ptyx if we define $\Supp_{D,X}:D(X)+1\to[\tr(D)]^{<\omega}$ by
\begin{equation*}
\Supp_{D,X}(\sigma):=\begin{cases}
\{(|a|,\sigma_0)\} & \text{if $\sigma\nf D(\iota_a^X\circ\en_a)(\sigma_0)\in D(X)$},\\
\emptyset & \text{if $\sigma=\top$}.
\end{cases}
\end{equation*}
If $D\cong D+1$ is a fixed point, then $D(0)\cong D(0)+1$ cannot be a well order, so that $D$ is no dilator. To see that~$D\mapsto D+1$ is not normal, we consider the constant dilators with values $D(X)=0$ and~$E(X)=1=\{0\}$. The unique mor\-phism $\nu:D\To E$ (with the empty function as components) is a segment. We~have
\begin{equation*}
0<_{(E+1)(0)}\top=(\nu+1)_0(\top)\in\rng((\nu+1)_0)
\end{equation*}
but $0\notin\rng((\nu+1)_0)$, which shows that $\nu+1$ is no segment.
\end{example}

In the rest of this section, we show that each preptyx~$P$ can be majorized by a normal preptyx~$P^*$, in the sense that there is a morphism $P(D)+1\To P^*(D+1)$ for each predilator~$D$. To motivate the summand~$1$, we recall the corresponding construction for functions from ordinals to ordinals (cf.~the introduction). If $h$ is any such function, then we can define a normal function~$g$ by setting
\begin{equation*}
g(\alpha):=\sum_{\gamma<\alpha}h(\gamma)+1,
\end{equation*}
or more formally $g(0)=0$, $g(\alpha+1)=g(\alpha)+h(\alpha)+1$ and $g(\lambda)=\sup_{\gamma<\lambda}g(\gamma)$ for~$\lambda$ limit. The successor clause reveals that $\gamma+1\leq\alpha$ implies $h(\gamma)+1\leq g(\alpha)$. On the level of dilators, the summand~$1$ cannot be avoided, at least not uniformly. To make this precise, let us say that a predilator $D$ is weakly normal if the range of $D(f):D(X)\to D(Y)$ is an initial segment of $D(Y)$ whenever the range of  the embedding $f:X\to Y$ is an initial segment of~$Y$ (cf.~the discussion at the beginning of the present section).

\begin{lemma}\label{lem:weakly-normal}
Assume that the predilator~$D$ is uniformly majorized by a weakly normal predilator~$E$, in the sense that there is a natural transformation~$\nu:D\To E$. Then $D$ itself must already be weakly normal.
\end{lemma}
\begin{proof}
Assume that the range of the embedding~$f:X\to Y$ is an initial segment of~$Y$, and consider an inequality $\sigma<_{D(Y)}D(f)(\tau)\in\rng(D(f))$. We get
\begin{equation*}
\nu_Y(\sigma)<_{E(Y)}\nu_Y\circ D(f)(\tau)=E(f)\circ\nu_X(\tau)\in\rng(E(f)).
\end{equation*}
Due to the assumption that~$E$ is weakly normal, this yields $\nu_Y(\sigma)=E(f)(\sigma_0)$ for some $\sigma_0\in E(X)$. Invoking Lemma~\ref{lem:transfos-respect-supp}, we get
\begin{equation*}
\supp^D_Y(\sigma)=\supp^E_Y\circ\nu_Y(\sigma)=\supp^E_Y\circ E(f)(\sigma_0)=[f]^{<\omega}\circ\supp^E_X(\sigma_0)\subseteq\rng(f).
\end{equation*}
Now the support condition from part~(ii) of Definition~\ref{def:dilator} yields $\sigma\in\rng(D(f))$, as needed to show that $D$ is weakly normal.
\end{proof}

Let us come back to the construction of~$P^*$. The inequality $\gamma<\alpha$ on the level of ordinals corresponds to a segment $D_0\To D$ on the level of predilators. Informally, we would like to reproduce the definition of~$g$ by setting
\begin{equation*}
P^*(D):=\sum\{P(D_0)+1\,|\,\text{there is a ``proper" segment~$D_0\To D$}\}.
\end{equation*}
A result of Girard (see e.\,g.~\cite[Lemma~2.11]{girard-normann85}) suggests a linear order on the summands. In order to make precise sense of our intuitive definition, we analyse the collection of segments~$D_0\To D$ in terms of the trace~$\tr(D)$.

\begin{definition}\label{def:segments-trace}
Given a predilator~$D$, we define a relation $\ll$ on the trace~$\tr(D)$ by stipulating that $(m,\sigma)\ll(n,\tau)$ holds if, and only if, we have
\begin{equation*}
D(f)(\sigma)<_{D(X)} D(g)(\tau)
\end{equation*}
for all embeddings $f:m\to X$ and $g:n\to X$ into a linear order~$X$.
\end{definition}

The following shows that the relation~$\ll$ is $\Delta^0_1$-definable. It also shows that it make no difference whether we consider~$D$ as a class-sized or as a coded dilator.

\begin{lemma}
We already have $(m,\sigma)\ll(n,\tau)$ if the condition from Definition~\ref{def:segments-trace} is satisfied for all embeddings into $X=m+n=\{0,\dots,m+n-1\}$.
\end{lemma}
\begin{proof}
Consider embeddings $f:m\to X$ and $g:n\to X$ into an arbitrary order~$X$. For $a=\rng(f)$ we have $f=\iota_a^X\circ\en_a$, as both functions are increasing and have the same finite range. Similarly, we have $g=\iota_b^X\circ\en_b$ for $b=\rng(g)$. Pick an embedding $h:|a\cup b|\to m+n$. The assumption of the present lemma yields
\begin{equation*}
D(h\circ|\iota_a^{a\cup b}|)(\sigma)<_{D(m+n)}D(h\circ|\iota_b^{a\cup b}|)(\tau).
\end{equation*}
This implies $D(|\iota_a^{a\cup b}|)(\sigma)<_{D(|a\cup b|)}D(|\iota_b^{a\cup b}|)(\tau)$ and then
\begin{multline*}
D(f)(\sigma)=D(\iota_{a\cup b}^X\circ\iota_a^{a\cup b}\circ\en_a)(\sigma)=D(\iota_{a\cup b}^X\circ\en_{a\cup b}\circ|\iota_a^{a\cup b}|)(\sigma)<_{D(X)}{}\\
{}<_{D(X)}D(\iota_{a\cup b}^X\circ\en_{a\cup b}\circ|\iota_b^{a\cup b}|)(\tau)=D(\iota_{a\cup b}^X\circ\iota_b^{a\cup b}\circ\en_b)(\tau)=D(g)(\tau),
\end{multline*}
as required by the full condition from Definition~\ref{def:segments-trace}.
\end{proof}

Let us verify a basic property:

\begin{lemma}
The relation $\ll$ on $\tr(D)$ is irreflexive and transitive.
\end{lemma}
\begin{proof}
To see $(m,\sigma)\not\ll(m,\sigma)$, it suffices to note that we have
\begin{equation*}
D(\iota_m^m)(\sigma)\not<_{D(m)}D(\iota_m^m)(\sigma),
\end{equation*}
since~$<_{D(m)}$ is irreflexive. For transitivity we consider inequalities $(m,\sigma)\ll(n,\tau)$ and $(n,\tau)\ll(k,\rho)$. Given embeddings $f:m\to X$ and $g:k\to X$, pick an order~$Y$ that is large enough to admit embeddings $h:X\to Y$ and $h':n\to Y$. We get
\begin{equation*}
D(h\circ f)(\sigma)<_{D(Y)} D(h')(\tau)<_{D(Y)} D(h\circ g)(\rho),
\end{equation*}
which implies $D(f)(\sigma)<_{D(X)} D(g)(\rho)$, as needed for $(m,\sigma)\ll(k,\rho)$.
\end{proof}

As promised, the relation $\ll$ on the trace can be used to characterize segments:

\begin{proposition}\label{prop:segment-trace}
For any morphism $\nu:D\To E$ between predilators~$D$ and~$E$, the following are equivalent:
\begin{enumerate}[label=(\roman*)]
\item the morphism $\nu:D\To E$ is a segment,
\item given $(m,\sigma)\in\rng(\nu)$ and $(m,\sigma)\not\ll(n,\tau)\in\tr(E)$, we get $(n,\tau)\in\rng(\nu)$.
\end{enumerate}
\end{proposition}
\begin{proof}
To show that~(i) implies~(ii), we assume $(m,\sigma)\in\rng(\nu)$ and $(m,\sigma)\not\ll(n,\tau)$. The former is witnessed by an element $\sigma_0\in D(m)$ with $\sigma=\nu_m(\sigma_0)$, while the latter yields embeddings $f:m\to X$ and $g:n\to X$ with
\begin{equation*}
E(g)(\tau)\leq_{E(X)} E(f)(\sigma)=E(f)\circ\nu_m(\sigma_0)=\nu_X\circ D(f)(\sigma_0)\in\rng(\nu_X).
\end{equation*}
Given that $\nu$ is a segment, we obtain $E(g)(\tau)=\nu_X(\rho)$ for some element $\rho\in D(X)$. Write $\rho\nf D(\iota_a^X\circ\en_a)(\rho_0)$ and observe
\begin{equation*}
E(g)(\tau)=\nu_X\circ D(\iota_a^X\circ\en_a)(\rho_0)=E(\iota_a^X\circ\en_a)\circ\nu_{|a|}(\rho_0).
\end{equation*}
Recall that $(n,\tau)\in\tr(E)$ entails $\supp^E_n(\tau)=n$. Together with Lemma~\ref{lem:transfos-respect-supp}, this means that $(|a|,\rho_0)\in\tr(D)$ yields $\supp^E_{|a|}\circ\nu_{|a|}(\rho_0)=\supp^D_{|a|}(\rho_0)=|a|$. We obtain
\begin{multline*}
[g]^{<\omega}(n)=\supp^E_X(E(g)(\tau))=\supp^E_X(\nu_X\circ D(\iota_a^X\circ\en_a)(\rho_0))=\\
=\supp^E_X(E(\iota_a^X\circ\en_a)\circ\nu_{|a|}(\rho_0))=[\iota_a^X\circ\en_a]^{<\omega}(\supp^E_{|a|}\circ\nu_{|a|}(\rho_0))=a,
\end{multline*}
so that $n=|a|$ and $g=\iota_a^X\circ\en_a$. By the above we get $\tau=\nu_{|a|}(\rho_0)$ and hence
\begin{equation*}
(n,\tau)=(|a|,\nu_{|a|}(\rho_0))=\tr(\nu)((|a|,\rho_0))\in\rng(\nu).
\end{equation*}
To show that~(ii) implies~(i), we consider $E(X)\ni\tau<_{E(X)}\nu_X(\sigma)$ with $\sigma\in D(X)$. Let us write $\tau\nf E(\iota_b^X\circ\en_b)(\tau_0)$ and $\sigma\nf D(\iota_a^X\circ\en_a)(\sigma_0)$. Then
\begin{equation*}
E(\iota_a^X\circ\en_a)\circ\nu_{|a|}(\sigma_0)=\nu_X(\sigma)\not<_{E(X)}\tau=E(\iota_b^X\circ\en_b)(\tau_0)
\end{equation*}
witnesses that we have
\begin{equation*}
\rng(\nu)\ni\tr(\nu)((|a|,\sigma_0))=(|a|,\nu_{|a|}(\sigma_0))\not\ll(|b|,\tau_0).
\end{equation*}
By~(ii) we get $(|b|,\tau_0)\in\rng(\nu)$, say $\tau_0=\nu_{|b|}(\tau_1)$ with $\tau_1\in D(|b|)$. This yields
\begin{equation*}
\tau=E(\iota_b^X\circ\en_b)\circ\nu_{|b|}(\tau_1)=\nu_X\circ D(\iota_b^X\circ\en_b)(\tau_1)\in\rng(\nu_X),
\end{equation*}
as needed to show that $\nu$ is a segment.
\end{proof}

The following result implies that incomparability under $\ll$ is an equivalence relation that is compatible with~$\ll$. Hence one can linearize $\ll$ by identifying incomparable elements.

\begin{lemma}\label{lem:linearize-ll}
Assume that we have $(m,\sigma)\ll (n,\tau)$ in $\tr(D)$.
\begin{enumerate}[label=(\alph*)]
\item If we have $(m,\sigma)\not\ll(m',\sigma')$, then we have $(m',\sigma')\ll (n,\tau)$.
\item If we have $(n',\tau')\not\ll(n,\tau)$, then we have $(m,\sigma)\ll (n',\tau')$.
\end{enumerate}
\end{lemma}
\begin{proof}
First note that the contrapositive of~(b) is an instance of~(a). In order to establish the latter, we deduce $(m,\sigma)\not\ll (n,\tau)$ from the assumption that we have $(m,\sigma)\not\ll(m',\sigma')$ and $(m',\sigma')\not\ll (n,\tau)$. This assumption yields embeddings $f:m\to X$, $g:m'\to X$, $f':m'\to Y$ and $g':n\to Y$ with
\begin{equation*}
D(g)(\sigma')\leq_{D(X)} D(f)(\sigma)\quad\text{and}\quad D(g')(\tau)\leq_{D(Y)} D(f')(\sigma').
\end{equation*}
Consider a linear order~$Z$ that is large enough to admit embeddings $h:X\to Z$ and $h':Y\to Z$ with $h'(y)\leq_Z h(x)$ for all $x\in X$ and $y\in Y$. This yields $h'\circ f'\leq h\circ g$, in the notation from the beginning of Section~\ref{sect:cat-dilators}. Since the predilator~$D$ is monotone (cf.~Definition~\ref{def:dilator}), we obtain
\begin{equation*}
D(h'\circ g')(\tau)\leq_{D(Z)} D(h'\circ f')(\sigma')\leq_{D(Z)} D(h\circ g)(\sigma')\leq_{D(Z)} D(h\circ f)(\sigma),
\end{equation*}
as needed to witness $(m,\sigma)\not\ll (n,\tau)$.
\end{proof}

For technical reasons, we will not identify incomparable elements. Instead, we linearize $\ll$ as follows (cf.~\cite[Lemma~2.11]{girard-normann85}): Each order $n=|n|=\{0,\dots,n-1\}$ can be identified with a suborder of~$\omega$. In view of Definition~\ref{def:extend-coded-dil}, this identification yields $\tr(D)\subseteq\overline D(\omega)$, with $D\!\restriction\!\lo_0$ at the place of~$D$ when the latter is class-sized. For elements $(m,\sigma)$ and $(n,\tau)$ of $\tr(D)$ we set
\begin{equation*}
(m,\sigma)<_{\tr(D)}(n,\tau)\quad:\Leftrightarrow\quad (m,\sigma)<_{\overline D(\omega)} (n,\tau).
\end{equation*}
Let us observe the following:

\begin{lemma}\label{lem:linearize-ll}
The relation $<_{\tr(D)}$ is a linear order. It is well founded if the same holds for $<_{\overline D(\omega)}$. Furthermore, we have
\begin{equation*}
(m,\sigma)\ll(n,\tau)\quad\Rightarrow\quad (m,\sigma)<_{\tr(D)}(n,\tau)
\end{equation*}
for any elements $(m,\sigma)$ and $(n,\tau)$ of $\tr(D)$.
\end{lemma}
\begin{proof}
Proposition~\ref{prop:coded-to-class-predil} entails that $\overline D(\omega)$ is a linear order, which yields the first part of the claim. The definition of $\ll$ reveals that $(m,\sigma)\ll(n,\tau)$ implies
\begin{equation*}
D(|\iota_m^{m\cup n}|)(\sigma)<_{D(|m\cup n|)} D(|\iota_n^{m\cup n}|)(\tau).
\end{equation*}
By Definition~\ref{def:extend-coded-dil} this amounts to $(m,\sigma)<_{\overline D(\omega)} (n,\tau)$, as required for the last part of the lemma. To avoid confusion, we point out that our general notation is unnecessarily complex for the present situation: the function $|\iota_m^{m\cup n}|=\iota_m^{m\cup n}$ is simply the inclusion of $m=\{0,\dots,m-1\}$ into $m\cup n=\max(m,n)=\{0,\dots,\max(m,n)-1\}$.
\end{proof}

As the following shows, each ``proper" segment $\nu:D\To E$ is determined (cf.~Pro\-po\-si\-tion~\ref{prop:trace-inclusion-to-morphism}) by an element~$\rho=(k,\rho_0)\in\tr(E)$. Usually, $\rho$ is not unique: if $\rho$ and $\rho'$ are incomparable, then they determine the same subset of~$\tr(E)$, by Lemma~\ref{lem:linearize-ll}.

\begin{corollary}
Consider a morphism $\nu:D\To E$ of predilators and assume that $\overline E(\omega)$ is well founded. If we have $\rng(\nu)\neq\tr(E)$, then the following are equivalent:
\begin{enumerate}[label=(\roman*)]
\item the morphism $\nu$ is a segment,
\item we have $\rng(\nu)=\{\sigma\in\tr(E)\,|\,\sigma\ll\rho\}$ for some $\rho\in\tr(E)$.
\end{enumerate}
\end{corollary}
\begin{proof}
In view of $\rng(\nu)\neq\tr(E)$ and the previous lemma, we may pick a $\rho\notin\rng(\nu)$ that is $\ll$-minimal. Assuming~(i), we show that any such~$\rho$ witnesses~(ii). To see that $\sigma\ll\rho$ implies $\sigma\in\rng(\nu)$, it suffices to invoke the minimality of~$\rho$. Now assume $\sigma\in\rng(\nu)$. If we had $\sigma\not\ll\rho$, then Proposition~\ref{prop:segment-trace} would yield $\rho\in\rng(\nu)$, against the choice of~$\rho$. Hence we must have $\sigma\ll\rho$. Let us now assume that~(ii) holds for some given $\rho\in\tr(E)$. By Lemma~\ref{lem:linearize-ll}, it follows that $\sigma\in\rng(\nu)$ and $\sigma\not\ll\tau$ imply $\tau\in\rng(\nu)$. Then Proposition~\ref{prop:segment-trace} yields~(i).
\end{proof}

According to Definition~\ref{def:trace-morphism}, each morphism $\nu:D\To E$ of predilators yields a function $\tr(\nu):\tr(D)\to\tr(E)$. We will need the following:

\begin{lemma}
Consider a morphism $\nu:D\To E$. If we have $\sigma<_{\tr(D)}\tau$, then we have $\tr(\nu)(\sigma)<_{\tr(E)}\tr(\nu)(\tau)$.
\end{lemma}
\begin{proof}
In defining $<_{\tr(D)}$, we have identified $\tr(D)$ with a subset of~$\overline D(\omega)$. Under this identification, the value $\tr(\nu)(n,\sigma_0)=(n,\nu_n(\sigma_0))$ from Definition~\ref{def:trace-morphism} coincides with the value $\overline\nu_\omega(n,\sigma_0)=(n,\nu_n(\sigma_0))$ from Definition~\ref{def:morphs-dils-extend}. From Lemma~\ref{lem:morphism-dils-extend} we know that the function $\overline\nu_\omega:\overline D(\omega)\to\overline E(\omega)$ is order preserving. Hence the same holds for the function $\tr(\nu)$.
\end{proof}

We now have all ingredients to make the definition of $P^*$ official (cf.~the discussion before Definition~\ref{def:segments-trace}). For $\rho\in\tr(D)$ we abbreviate
\begin{equation*}
D[\rho]:=D[\{\sigma\in\tr(D)\,|\,\sigma\ll\rho\}],
\end{equation*}
where the right side is explained by Definition~\ref{def:D[A]}. Together with the notation from Example~\ref{ex:ptyx-succ}, this allows us to describe the action of $P^*$ on predilators:

\begin{definition}\label{def:P*-objects}
Consider a $2$-preptyx~$P$. For each predilator~$D$ and each order~$X$, we consider the set
\begin{equation*}
P^*(D)(X):=\sum_{\rho\in\tr(D)}(P(D[\rho])+1)(X),
\end{equation*}
which has elements $(\rho,\sigma)$ with $\rho\in\tr(D)$ and $\sigma\in P(D[\rho])(X)\cup\{\top\}$. We set
\begin{equation*}
(\rho,\sigma)<_{P^*(D)(X)}(\rho',\sigma')\quad:\Leftrightarrow\quad \rho<_{\tr(D)}\rho'\text{ or }(\rho=\rho'\text{ and }\sigma<_{P(D[\rho])(X)+1}\sigma').
\end{equation*}
Given an embedding $f:X\to Y$, we define $P^*(D)(f):P^*(D)(X)\to P^*(D)(Y)$ by
\begin{equation*}
P^*(D)(f)(\rho,\sigma):=(\rho,(P(D[\rho])+1)(f)(\sigma)).
\end{equation*}
Finally, we set
\begin{equation*}
\supp^{P^*(D)}_X(\rho,\sigma):=\supp^{P(D[\rho])+1}_X(\sigma)
\end{equation*}
to define a family of functions $\supp^{P^*(D)}_X:P^*(D)(X)\to[X]^{<\omega}$.
\end{definition}

By the discussion after Definition~\ref{def:ptyx}, the preptyx~$P$ comes with $\Delta^0_1$-definitions of $\sigma\in P(D)(X)$ and other relevant relations. These definitions refer to the coded predilator~$D\subseteq\mathbb N$ as a set parameter: they involve subformulas like $\tau\in D(n)$, or rather $(0,n,\sigma)\in D$, in view of the discussion after Definition~\ref{def:coded-predil}. As we have observed after Definition~\ref{def:D[A]}, the relation $\tau\in D[A](X)$ is $\Delta^0_1$-definable relative to $D\subseteq\mathbb N$ and $A\subseteq\tr(D)$. In order to turn the given definition of~$\sigma\in P(D)$ into a $\Delta^0_1$-definition of $\sigma\in P(D[\rho])$, it suffices to replace subformulas like $\tau\in D(n)$ by corresponding $\Delta^0_1$-formulas like $\tau\in D[\rho](n)$. Once we have definitions for $P(D[\rho])$, it is straightforward to construct $\Delta^0_1$-definitions of $(\rho,\sigma)\in P^*(D)(X)$ and the other relations that constitute~$P^*(D)$. The following results will allow us to define the action of~$P^*$ on morphisms of predilators.

\begin{lemma}\label{lem:tr-equiv-ll}
Consider a morphism $\nu:D\To E$. We have $\sigma\ll\rho$ in $\tr(D)$ if, and only if, we have $\tr(\nu)(\sigma)\ll\tr(\nu)(\rho)$ in $\tr(E)$.
\end{lemma}
\begin{proof}
Write $\sigma=(m,\sigma_0)$ and $\rho=(n,\rho_0)$, which yields $\tr(\nu)(\sigma)=(m,\nu_m(\sigma_0))$ and $\tr(\nu)(\rho)=(n,\nu_n(\rho_0))$. In view of Definition~\ref{def:segments-trace}, it suffices to observe that the inequality $D(f)(\sigma_0)<_{D(X)} D(g)(\rho_0)$ is equivalent to
\begin{equation*}
E(f)\circ\nu_m(\sigma_0)=\nu_X\circ D(f)(\sigma_0)<_{E(X)}\nu_X\circ D(g)(\rho_0)=E(g)\circ\nu_n(\rho_0),
\end{equation*}
for arbitrary embeddings $f:m\to X$ and $g:n\to X$.
\end{proof}

For each element $\rho\in\tr(D)$, Lemma~\ref{lem:iota[A]} yields a morphism
\begin{equation*}
\iota[\rho]:=\iota[D,\rho]:=\iota[\{\sigma\in\tr(D)\,|\,\sigma\ll\rho\}]:D[\rho]\To D
\end{equation*}
with $\rng(\iota[\rho])=\{\sigma\in\tr(D)\,|\,\sigma\ll\rho\}$.

\begin{corollary}\label{cor:morph-trace-restriction}
For each morphism $\nu:D\To E$ and each $\rho\in\tr(D)$ we have
\begin{equation*}
\rng(\nu\circ\iota[D,\rho])\subseteq\rng(\iota[E,\tr(\nu)(\rho)]).
\end{equation*}
If $\nu$ is a segment, then the converse inclusion holds as well. 
\end{corollary}
\begin{proof}
An arbitrary element of $\rng(\nu\circ\iota[\rho])$ has the form $\tr(\nu)(\sigma)$ with $\sigma\in\rng(\iota[\rho])$ and hence $\sigma\ll\rho$. By the previous lemma we get $\tr(\nu)(\sigma)\ll\tr(\nu)(\rho)$, as needed for $\tr(\nu)(\sigma)\in\rng(\iota[\tr(\nu)(\rho)])$. Now assume that $\nu$ is a segment. For an arbitrary element $\tau\in\rng(\iota[\tr(\nu)(\rho)])$ we have $\tau\ll\tr(\nu)(\rho)$ and hence
\begin{equation*}
\rng(\nu)\ni\tr(\nu)(\rho)\not\ll\tau.
\end{equation*}
By Proposition~\ref{prop:segment-trace} we get $\tau\in\rng(\nu)$, say $\tau=\tr(\nu)(\sigma)$. The previous lemma yields $\sigma\ll\rho$ and hence $\sigma\in\rng(\iota[\rho])$, so that we get $\tau=\tr(\nu)(\sigma)\in\rng(\nu\circ\iota[\rho])$.
\end{proof}

Together with Proposition~\ref{prop:trace-inclusion-to-morphism}, the corollary justifies the following:

\begin{definition}\label{def:nu-rho}
For a morphism $\nu:D\To E$ and an element $\rho\in\tr(D)$, we define
\begin{equation*}
\nu^\rho:D[\rho]\To E[\tr(\nu)(\rho)]
\end{equation*}
as the unique morphism with $\iota[E,\tr(\nu)(\rho)]\circ\nu^\rho=\nu\circ\iota[D,\rho]$.
\end{definition}

Let us observe that $\nu^\rho$ is $\Delta^0_1$-definable, given that the same holds for $\nu$: According to Definition~\ref{def:iota[A]}, the components of $\iota[\rho]$ are inclusion maps. For $\sigma\in D[\rho](X)$, this means that $\nu^\rho_X(\sigma)=\tau$ is equivalent to $\nu_X(\sigma)=\tau$. The following result will be needed to verify the support condition for~$P^*$. In view of Proposition~\ref{prop:pullback-dilator}, it tells us that we are concerned with a pullback.

\begin{corollary}\label{cor:nu-rho-pullback}
We have $\rng(\nu)\cap\rng(\iota[E,\tr(\nu)(\rho)])\subseteq\rng(\nu\circ\iota[D,\rho])$ for each morphism $\nu:D\To E$ and each element~$\rho\in\tr(D)$.
\end{corollary}
\begin{proof}
An arbitrary element of $\rng(\nu)\cap\rng(\iota[\tr(\nu)(\rho)])$ has the form $\tr(\nu)(\sigma)$ and satisfies $\tr(\nu)(\sigma)\ll\tr(\nu)(\rho)$. By Lemma~\ref{lem:tr-equiv-ll} we can conclude $\sigma\ll\rho$. The latter amounts to $\sigma\in\rng(\iota[\rho])$, which implies $\tr(\nu)(\sigma)\in\rng(\nu\circ\iota[\rho])$.
\end{proof}

Let us also record the following fact, which will be used to show that the preptyx~$P^*$ is normal (cf.~Definition~\ref{def:ptyx-normal}).

\begin{corollary}\label{cor:nu-rho-iso}
If $\nu:D\To E$ is a segment, then $\nu^\rho:D[\rho]\To E[\tr(\nu)(\rho)]$ is an isomorphism, for each element $\rho\in\tr(D)$.
\end{corollary}
\begin{proof}
Due to the second part of Corollary~\ref{cor:morph-trace-restriction}, we can invoke Proposition~\ref{prop:trace-inclusion-to-morphism} to get a morphism $\mu^\rho:E[\tr(\nu)(\rho)]\To D[\rho]$ with $\nu\circ\iota[\rho]\circ\mu^\rho=\iota[\tr(\nu)(\rho)]$. In view of
\begin{equation*}
\nu\circ\iota[\rho]\circ\mu^\rho\circ\nu^\rho=\iota[\tr(\nu)(\rho)]\circ\nu^\rho=\nu\circ\iota[\rho],
\end{equation*}
the uniqueness part of Proposition~\ref{prop:trace-inclusion-to-morphism} implies that $\mu^\rho\circ\nu^\rho$ is the identity on~$D[\rho]$. An analogous argument shows that $\nu^\rho\circ\mu^\rho$ is the identity on $E[\tr(\nu)(\rho)]$.
\end{proof}

In order to turn $P^*$ into a preptyx, we will define support functions $\Supp^*_{D,X}$ as in Lemma~\ref{lem:ptyx-support-variant-back}. By Lemma~\ref{lem:ptyx-support-variant}, we may assume that the preptyx~$P$ comes with functions $\Supp_{D,X}:P(D)(X)\to[\tr(D)]^{<\omega}$.

\begin{definition}
Consider a preptyx~$P$. For each morphism $\nu:D\To E$ and each linear order~$X$, we define $P^*(\nu)_X:P^*(D)(X)\To P^*(E)(X)$ by
\begin{equation*}
P^*(\nu)_X(\rho,\sigma):=\begin{cases}
(\tr(\nu)(\rho),P(\nu^\rho)_X(\sigma)) & \text{if $\sigma\in P(D[\rho])(X)$},\\
(\tr(\nu)(\rho),\top) & \text{if $\sigma=\top$}.
\end{cases}
\end{equation*}
In order to define functions $\Supp^*_{D,X}:P^*(D)(X)\to[\tr(D)]^{<\omega}$, we set
\begin{equation*}
\Supp^*_{D,X}(\rho,\sigma):=\begin{cases}
\{\rho\}\cup\Supp_{D,X}(P(\iota[\rho])_X(\sigma)) & \text{if $\sigma\in P(D[\rho])(X)$},\\
\{\rho\} & \text{if $\sigma=\top$},
\end{cases}
\end{equation*}
for each predilator~$D$ and each linear order~$X$.
\end{definition}

The relations $P^*(\nu)_X(\sigma)=\tau$ and $\Supp^*_{D,X}(\sigma)=a$ are $\Delta^0_1$-definable, by the discussion after Definition~\ref{def:P*-objects}. Our constructions culminate in the following result:

\begin{theorem}\label{thm:ptyx-normal}
For any $2$-preptyx $P$ we have the following:
\begin{enumerate}[label=(\alph*)]
\item $P^*$ is a normal $2$-preptyx,
\item if $P$ is a $2$-ptyx (i.\,e.,~preserves dilators), then so is~$P^*$,
\item there is a family of morphisms $\xi^D:P(D)+1\To P^*(D+1)$ that is natural in the predilator~$D$, in the sense that we have $\xi^E\circ (P(\nu)+1)=P^*(\nu+1)\circ\xi^D$ for any morphism $\nu:D\To E$.
\end{enumerate}
\end{theorem}
\begin{proof}
(a) Given a predilator~$D$, we can invoke Lemma~\ref{lem:D[A]-predil} to learn that $D[\rho]$ is a predilator for each $\rho\in\tr(D)$. Since $P$ is a preptyx, it follows that $P(D[\rho])$ and $P(D[\rho])+1$ are predilators. It is straightforward to deduce that $P^*(D)$ is a predilator as well. Similarly, $P^*(\nu):P^*(D)\To P^*(E)$ is a morphism of predilators when the same holds for~$\nu:D\To E$. The claim that $P^*$ is functorial can be reduced to the following facts: First, if $\nu:D\To D$ is the identity, then $\nu^\rho:D[\rho]\To D[\rho]$ is the identity for each~$\rho\in\tr(D)$. Secondly, we have
\begin{equation*}
(\mu\circ\nu)^\rho=\mu^{\tr(\nu)(\rho)}\circ\nu^\rho
\end{equation*}
for arbitrary morphisms $\nu:D_0\To D_1$ and $\mu:D_1\To D_2$ and any $\rho\in\tr(D_0)$. Both facts follow from the uniqueness part of Proposition~\ref{prop:trace-inclusion-to-morphism}. To show that the functions $\Supp^*_{D,X}:P^*(D)(X)\to[\tr(D)]^{<\omega}$ are natural in $D$, we consider a morphism $\nu:D\To E$ and an element $(\rho,\sigma)\in P^*(D)(X)$. Since $P$ is a preptyx, the support functions $\Supp_{D,X}:P(D)(X)\to[\tr(D)]^{<\omega}$ are natural (cf.~Lemma~\ref{lem:ptyx-support-variant}). In the more difficult case of $\sigma\neq\top$, we can deduce
\begin{align*}
\Supp^*_{E,X}(P^*(\nu)_X(\rho,\sigma))&=\Supp^*_{E,X}(\tr(\nu)(\rho),P(\nu^\rho)_X(\sigma))=\\
{}&=\{\tr(\nu)(\rho)\}\cup\Supp_{E,X}(P(\iota[\tr(\nu)(\rho)]\circ\nu^\rho)_X(\sigma))=\\
{}&=\{\tr(\nu)(\rho)\}\cup\Supp_{E,X}(P(\nu\circ\iota[\rho])_X(\sigma))=\\
{}&=\{\tr(\nu)(\rho)\}\cup[\tr(\nu)]^{<\omega}(\Supp_{D,X}(P(\iota[\rho])_X(\sigma)))=\\
{}&=[\tr(\nu)]^{<\omega}(\Supp^*_{D,X}(\rho,\sigma)).
\end{align*}
In order to conclude that $P^*$ is a preptyx, we verify the support condition from Lemma~\ref{lem:ptyx-support-variant} (which implies the one from Definition~\ref{def:ptyx}, by Lemma~\ref{lem:ptyx-support-variant-back}). For this purpose, we consider a morphism $\nu:D\To E$ and an element $(\rho,\sigma)\in P^*(E)(X)$ with $\Supp^*_{E,X}(\rho,\sigma)\subseteq\rng(\nu)$. The latter yields $\rho\in\rng(\nu)$, say $\rho=\tr(\nu)(\rho_0)$. In case of $\sigma\neq\top$, it also yields
\begin{multline*}
\Supp_{E,X}(P(\iota[\rho])_X(\sigma))=[\tr(\iota[\rho])]^{<\omega}(\Supp_{D,X}(\sigma))\subseteq{}\\
{}\subseteq\rng(\nu)\cap\rng(\iota[\rho])\subseteq\rng(\nu\circ\iota[\rho_0]),
\end{multline*}
where the last inclusion relies on Corollary~\ref{cor:nu-rho-pullback}. By the support condition for the preptyx~$P$, we get an element $\sigma_0\in P(D[\rho_0])(X)$ with
\begin{equation*}
P(\iota[\rho])_X(\sigma)=P(\nu\circ\iota[\rho_0])_X(\sigma_0)=P(\iota[\rho]\circ\nu^{\rho_0})_X(\sigma_0).
\end{equation*}
Since $P(\iota[\rho])$ is a morphism of predilators, the component $P(\iota[\rho])_X$ is an order embedding and in particular injective. We obtain $\sigma=P(\nu^{\rho_0})_X(\sigma_0)$ and then
\begin{equation*}
(\rho,\sigma)=(\tr(\nu)(\rho_0),P(\nu^{\rho_0})_X(\sigma_0))=P^*(\nu)_X(\rho_0,\sigma_0)\in\rng(P^*(\nu)_X),
\end{equation*}
as required by the support condition for~$P^*$. Finally, we show that the $2$-preptyx~$P^*$ is normal (cf.~Definition~\ref{def:ptyx-normal}). Consider a segment $\nu:D\To E$ and an inequality
\begin{equation*}
(\rho,\sigma)<_{P^*(E)(X)} P^*(\nu)_X(\rho',\sigma')\in\rng(P^*(\nu)_X).
\end{equation*}
The latter entails $\rho\leq_{\tr(E)}\tr(\nu)(\rho')$, so that Lemma~\ref{lem:linearize-ll} yields
\begin{equation*}
\rng(\nu)\ni\tr(\nu)(\rho')\not\ll\rho.
\end{equation*}
By Proposition~\ref{prop:segment-trace} we get $\rho\in\rng(\nu)$, say $\rho=\tr(\nu)(\rho_0)$. Now Corollary~\ref{cor:nu-rho-iso} tells us that $\nu^{\rho_0}:D[\rho_0]\To E[\rho]$ is an isomorphism. Since~$P$ is a functor, it follows that the natural transformation $P(\nu^{\rho_0})$ and its component $P(\nu^{\rho_0})_X$ are isomorphisms as well. For $\sigma\neq\top$ we get $\sigma=P(\nu^{\rho_0})_X(\sigma_0)$ for some $\sigma_0\in P(D[\rho_0])_X$. This yields
\begin{equation*}
(\rho,\sigma)=(\tr(\nu)(\rho_0),P(\nu^{\rho_0})_X(\sigma_0))=P^*(\nu)_X(\rho_0,\sigma_0)\in\rng(P^*(\nu)_X).
\end{equation*}
For $\sigma=\top$ we have $(\rho,\sigma)=(\tr(\nu)(\rho_0),\top)=P^*(\nu)_X(\rho_0,\top)\in\rng(P^*(\nu)_X)$.

(b)~Consider a dilator~$D$ and a well order~$X$. We need to show that $P^*(D)(X)$ is well founded. The morphisms $\iota[\rho]:D[\rho]\To D$ ensure that $D[\rho]$ is a dilator for each~$\rho\in\tr(D)$. Given that~$P$ is a ptyx, it follows that each order $P(D[\rho])(X)+1$ is well founded. Now consider a (not necessarily strictly) descending sequence
\begin{equation*}
(\rho_0,\sigma_0),(\rho_1,\sigma_1),\ldots\subseteq P^*(D)(X).
\end{equation*}
The first components form a descending sequence with respect to the order~$<_{\tr(D)}$ on the trace of~$D$. Given that~$D$ is a dilator, we can invoke Lemma~\ref{lem:linearize-ll} to learn that~$<_{\tr(D)}$ is a well order. Hence there is some $N\in\mathbb N$ with $\rho_n=\rho_N$ for all~$n\geq N$. Then $\sigma_N,\sigma_{N+1},\ldots$ is a sequence in $P(D[\rho_N])(X)+1$. We have already seen that the latter is well founded. This yields an $n\geq N$ with $\sigma_n\leq_{P(D[\rho_N])(X)+1}\sigma_{n+1}$. Together with $\rho_n=\rho_{n+1}$ we get
\begin{equation*}
(\rho_n,\sigma_n)\leq_{P^*(D)(X)}(\rho_{n+1},\sigma_{n+1}),
\end{equation*}
as needed to show that $P^*(D)(X)$ is a well order.

(c)~The element $\top\in (D+1)(0)$ yields an element $(0,\top)\in\tr(D+1)$, which gives rise to a morphism $\iota[D+1,(0,\top)]:(D+1)[(0,\top)]\To D+1$ (cf.~the discussion before Corollary~\ref{cor:morph-trace-restriction}). We also have a morphism $\pi^D:D\To D+1$, where each component $\pi^D_X:D(X)\hookrightarrow D(X)\cup\{\top\}=(D+1)(X)$ is the obvious inclusion map. Let us show that we have
\begin{equation*}
\rng(\pi^D)\subseteq\rng(\iota[D+1,(0,\top)])=\{\sigma\in\tr(D+1)\,|\,\sigma\ll(0,\top)\}.
\end{equation*}
An arbitrary element of $\rng(\pi^D)$ can be written as $\tr(\pi^D)(m,\sigma_0)=(m,\pi^D_m(\sigma_0))$ for some $(m,\sigma_0)\in\tr(D)$. To show that we have $(m,\pi^D_m(\sigma_0))\ll(0,\top)$ in~$\tr(D+1)$, we consider arbitrary embeddings $f:m\to X$ and $g:0\to X$ into some order~$X$ (in~fact, $g$ can only be the empty function). In view of $\pi^D_m(\sigma_0)\in D(m)\subseteq D(m)\cup\{\top\}$, we have $(D+1)(f)(\pi^D_m(\sigma_0))=D(f)(\pi^D_m(\sigma_0))\in D(X)\subseteq D(X)\cup\{\top\}$. This yields
\begin{equation*}
(D+1)(f)(\pi^D_m(\sigma_0))<_{(D+1)(X)}\top=(D+1)(g)(\top),
\end{equation*}
as needed for $(m,\pi^D_m(\sigma_0))\ll(0,\top)$. Now Proposition~\ref{prop:trace-inclusion-to-morphism} yields a morphism
\begin{equation*}
\iota^D:D\To (D+1)[(0,\top)]\quad\text{with}\quad\iota[D+1,(0,\top)]\circ\iota^D=\pi^D.
\end{equation*}
Since $\iota[D+1,(0,\top)]$ and $\pi^D$ are $\Delta^0_1$-definable relative to~$D$, the same holds for~$\iota^D$. As~explained after Definition~\ref{def:P*-objects}, we obtain a $\Delta^0_1$-definition of the morphism
\begin{equation*}
P(\iota^D):P(D)\To P((D+1)[(0,\top)]).
\end{equation*}
The components of $\xi^D:P(D)+1\To P^*(D+1)$ can now be defined by
\begin{equation*}
\xi^D_X(\sigma):=\begin{cases}
((0,\top),P(\iota^D)_X(\sigma)) & \text{if $\sigma\in P(D)(X)\subseteq (P(D)+1)(X)$},\\
((0,\top),\sigma) & \text{if $\sigma=\top\in (P(D)+1)(X)$}.
\end{cases}
\end{equation*}
It is straightforward to verify that $\xi^D_X$ is an embedding and natural in~$X$, so that $\xi^D$ is a morphism of predilators. To establish naturality in~$D$, we consider $\nu:D\To E$. Note $\tr(\nu+1)(0,\top)=(0,(\nu+1)_0(\top))=(0,\top)$ and invoke Definition~\ref{def:nu-rho} to get
\begin{multline*}
\iota[E+1,(0,\top)]\circ(\nu+1)^{(0,\top)}\circ\iota^D=(\nu+1)\circ\iota[D+1,(0,\top)]\circ\iota^D=\\
=(\nu+1)\circ\pi^D=\pi^E\circ\nu=\iota[E+1,(0,\top)]\circ\iota^E\circ\nu.
\end{multline*}
This entails $(\nu+1)^{(0,\top)}\circ\iota^D=\iota^E\circ\nu$, since the components of our morphisms are embeddings. For an element $\sigma\in P(D)(X)\subseteq (P(D)+1)(X)$ we can deduce
\begin{align*}
\xi^E_X\circ(P(\nu)+1)_X(\sigma)&=\xi^E_X(P(\nu)_X(\sigma))=((0,\top),P(\iota^E)_X\circ P(\nu)_X(\sigma))=\\
{}&=(\tr(\nu+1)(0,\top),P((\nu+1)^{(0,\top)})_X\circ P(\iota^D)_X(\sigma))=\\
{}&=P^*(\nu+1)_X((0,\top),P(\iota^D)_X(\sigma))=P^*(\nu+1)_X\circ\xi^D_X(\sigma).
\end{align*}
The case of $\sigma=\top\in (P(D)+1)(X)$ is similar and easier.
\end{proof}

\section{From fixed points of $2$-ptykes to $\Pi^1_2$-induction}\label{sect:fixed-point-to-induction}

In this section, we deduce $\Pi^1_2$-induction along~$\mathbb N$ from the assumption that every normal $2$-ptyx has a fixed point that is a dilator. For this purpose, we work out the details of the argument that we have sketched in the introduction.

Consider a $\Pi^1_2$-formula $\psi(n)$ with a distinguished number variable, possibly with further number or set parameters. The Kleene normal form theorem (see e.\,g.~\cite[Lemma~V.1.4]{simpson09}) yields a $\Delta^0_0$-formula $\theta$ such that $\aca_0$ proves
\begin{equation*}
\psi(n)\leftrightarrow\forall_{Z\subseteq\mathbb N}\exists_{f:\mathbb N\to\mathbb N}\forall_{m\in\mathbb N}\,\theta(Z[m],f[m],n).
\end{equation*}
Here $f[m]=\langle f(0),\dots,f(m-1)\rangle$ denotes the sequence of the first $m$ values of~$f$, coded by a natural number; in writing $Z[m]$, we identify the set $Z\subseteq\mathbb N$ with its characteristic function. The formulas $\psi$ and $\theta$ will be fixed throughout the following. We also fix~$\aca_0$ as base theory.

Let us introduce some notation: We write $Y^{<\omega}$ for the set of finite sequences with entries from the set~$Y$. In order to refer to the entries of a sequence $s\in Y^{<\omega}$ of length $m=\len(s)$, we will often write it as $s=\langle s(0),\dots,s(m-1)\rangle$. For $k\leq m$ we put $s[k]:=\langle s(0),\dots,s(k-1)\rangle$. If $Y$ is linearly ordered, the Kleene-Brouwer order (also called Lusin-Sierpi\'nski order) on $Y^{<\omega}$ is the linear order defined by
\begin{equation*}
s<_{\kb(Y)}t\,:\Leftrightarrow\,\begin{cases}
\text{either }\len(s)>\len(t)\text{ and }s[\len(t)]=t,\\
\text{or $s[k]=t[k]$ and $s(k)<_Y t(k)$ for some $k<\min\{\len(s),\len(t)\}$.}
\end{cases}
\end{equation*}
We say that a subset $\mathcal T\subseteq Y^{<\omega}$ is a tree if $s\in\mathcal T$ and $k\leq\len(s)$ imply $s[k]\in\mathcal T$. Unless indicated otherwise, we assume that any tree $\mathcal T\subseteq Y^{<\omega}$ carries the Kleene-Brouwer order~$<_{\kb(Y)}$. Recall that a branch of $\mathcal T$ is given by a function $f:\mathbb N\to Y$ such that $f[m]\in\mathcal T$ holds for all~$m\in\mathcal T$. If $Y$ is a well order, then $\mathcal T\subseteq Y^{<\omega}$  (with order relation $<_{\kb(Y)}$) is well founded if, and only if, it has no branch (by the proof of~\cite[Lemma~V.1.3]{simpson09}, which is formulated for~\mbox{$Y=\mathbb N$}). One can conclude that our $\Pi^1_2$-formula $\psi(n)$ fails for $n\in\mathbb N$ if, and only if, there is a $Z\subseteq\mathbb N$ such that the Kleene-Brouwer order $<_{\kb(\mathbb N)}$ is well founded on the tree
\begin{equation*}
\mathcal T_Z^n:=\{t\in\mathbb N^{<\omega}\,|\,\forall_{k\leq\len(t)}\neg\theta(Z[k],t[k],n)\}.
\end{equation*}
In the following we construct predilators~$D_\psi^n$ such that $D_\psi^n$ is a dilator if, and only if, the instance $\psi(n)$ holds. This is a version of Girard's result that the notion of dilator is $\Pi^1_2$-complete. The construction that we present is due to D.~Normann (see~\cite[Theorem~8.E.1]{girard-book-part2}). We recall it in full detail, because the rest of this section depends on the construction itself, not just on the result. Given a linear order~$X$, the idea is to define $D^n_\psi(X)$ as a subtree of $(2\times X)^{<\omega}$ (recall $m=\{0,\dots,m-1\}$ with the usual linear order). Along each potential branch of $D^n_\psi(X)$, we aim to construct a set $Z\subseteq\mathbb N$ (determined by the characteristic function $\mathbb N\to 2$ from the first component of the branch) and, simultaneously, an embedding of $\mathcal T^n_Z$ into $X$. If $D^n_\psi$ fails to be a dilator, then~$D^n_\psi(X)$ has a branch for some well order~$X$. The resulting embedding $\mathcal T^n_Z\to X$ ensures that $\mathcal T^n_Z$ is a well order, so that $\psi(n)$ fails. Since~$Z$ is not given in advance, we need to approximate $\mathcal T^n_Z$ by the trees
\begin{equation*}
\mathcal T^n_s:=\{t\in\mathbb N^{<\omega}\,|\,\len(t)\leq\len(s)\text{ and }\forall_{k\leq\len(t)}\neg\theta(s[k],t[k],n)\}
\end{equation*}
for $s\in 2^{<\omega}$. The relation $t\in\mathcal T_s^n$ is $\Delta^0_1$-definable, as $\theta$ is a $\Delta^0_0$-formula. We observe $\mathcal T_{s[m]}^n=\{t\in\mathcal T_s^n\,|\,\len(t)\leq m\}$ for $m\leq\len(s)$, as well as $\mathcal T_Z^n=\bigcup\{\mathcal T_{Z[m]}^n\,|\,m\in\mathbb N\}$.

In order to describe the following constructions in an efficient way, we introduce some notation in connection with products. First recall that the product of sets $Y_0,\dots,Y_{k-1}$ is given by
\begin{equation*}
Y_0\times\dots\times Y_{k-1}:=\{(y_0,\dots,y_{k-1})\,|\,y_i\in Y_i\text{ for each }i<k\}.
\end{equation*}
If $Y_i=(Y_i,<_i)$ is a linear order for each $i<k$, then we assume that the product carries the lexicographic order, in which $(y_0,\dots,y_{k-1})$ precedes $(y_0',\dots,y_{k-1}')$ if, and only if, there is an index $j<k$ with $y_j<_jy_j'$ and $y_i=y_i'$ for all~$i<j$. Given sequences $s_i=\langle s_i(0),\dots,s_i(m-1)\rangle\in Y_i^{<\omega}$ of the same length, we can construct a sequence $s_0\times\dots\times s_{k-1}\in(Y_0\times\dots\times Y_{k-1})^{<\omega}$ by setting
\begin{equation*}
s_0\times\dots\times s_{k-1}:=\langle (s_0(0),\dots,s_{k-1}(0)),\dots,(s_0(m-1),\dots,s_{k-1}(m-1))\rangle.
\end{equation*}
Note that any sequence in $(Y_0\times\dots\times Y_{k-1})^{<\omega}$ can be uniquely written in this form. Analogously, we will combine functions $g_i:\mathbb N\to Y_i$ into a function
\begin{equation*}
g_0\times\dots\times g_{k-1}:\mathbb N\to Y_0\times\dots\times Y_{k-1}
\end{equation*}
with $(g_0\times\dots\times g_{k-1})(l):=(g_0(l),\dots,g_{k-1}(l))$. For arbitrary $m\in\mathbb N$, we can observe
\begin{equation*}
(g_0\times\dots\times g_{k-1})[m]=g_0[m]\times\dots\times g_{k-1}[m].
\end{equation*}
We use our product notation in a somewhat flexible way, for example by combining $s_0\in Y_0^{<\omega}$ and $t=s_1\times s_2\in(Y_1\times Y_2)^{<\omega}$ into $s_0\times t:=s_0\times s_1\times s_2\in(Y_0\times Y_1\times Y_2)^{<\omega}$. We can now officially define the predilators~$D^n_\psi$ that were mentioned above.

\begin{definition}\label{def:dils-D_n}
For each linear order~$X$, we define $D^n_\psi(X)$ as the tree of all sequences $s\times t\in(2\times X)^{<\omega}$ with the following property: For all indices $i,j<\len(t)$ that are (numerical codes for) elements of $\mathcal T_s^n\subseteq\mathbb N^{<\omega}$, we have
\begin{equation*}
i<_{\kb(\mathbb N)}j\quad\To\quad t(i)<_X t(j).
\end{equation*}
Given an embedding $f:X\to Y$, we define $D^n_\psi(f):D^n_\psi(X)\to D^n_\psi(Y)$ by
\begin{equation*}
D^n_\psi(f)(s\times\langle t(0),\dots,t(m-1)\rangle):=s\times\langle f(t(0)),\dots,f(t(m-1))\rangle.
\end{equation*}
Finally, we define functions $\supp^n_X:D^n_\psi(X)\to[X]^{<\omega}$ by setting
\begin{equation*}
\supp^n_X(s\times\langle t(0),\dots,t(m-1)\rangle):=\{t(0),\dots,t(m-1)\}
\end{equation*}
for each linear order~$X$.
\end{definition}

As mentioned before, the following comes from Normann's proof of Girard's result that the notion of dilator is~$\Pi^1_2$-complete (see~\cite[Theorem~8.E.1]{girard-book-part2}).

\begin{proposition}\label{prop:D_n-psi}
For each $n\in\mathbb N$, we have the following:
\begin{enumerate}[label=(\alph*)]
\item the constructions from Definition~\ref{def:dils-D_n} yield a predilator~$D^n_\psi$,
\item the predilator~$D^n_\psi$ is a dilator if, and only if, the $\Pi^1_2$-formula $\psi(n)$ holds.
\end{enumerate}
\end{proposition}
\begin{proof}
(a) It is straightforward to check that $D^n_\psi$ is a monotone endofunctor of linear orders and that $\supp^n:D^n_\psi\To[\cdot]^{<\omega}$ is a natural transformation. To verify the support condition from part~(ii) of Definition~\ref{def:dilator}, we consider an embedding $f:X\to Y$ and an element $s\times t=s\times\langle t(0),\dots,t(m-1)\rangle\in D_n(Y)$ with
\begin{equation*}
\{t(0),\dots,t(m-1)\}=\supp^n_X(s\times t)\subseteq\rng(f).
\end{equation*}
Define $t'=\langle t'(0),\dots,t'(m-1)\rangle\in X^{<\omega}$ by stipulating $f(t'(i))=t(i)$ for $i<m$. To conclude $s\times t=D^n_\psi(f)(s\times t')\in\rng(D_n(f))$, it remains to check $s\times t'\in D^n_\psi(X)$. For this purpose, we consider indices $i,j<m$ with $i,j\in\mathcal T_s^n$ and $i<_{\kb(\mathbb N)}j$. In view of $s\times t\in D^n_\psi(Y)$ we have $t(i)<_Y t(j)$. The required inequality $t'(i)<_Xt'(j)$ follows because $f$ is an embedding.

(b) Both directions are established by contraposition. First assume that $\psi(n)$ fails. We can then consider a set~$Z\subseteq\mathbb N$ such that $\mathcal T_Z^n$ is well founded. Write $g:\mathbb N\to 2$ for the characteristic function of~$Z$, and put $X:=\mathcal T_Z^n\cup\{\top\}$ with a new maximal element~$\top$. Using the latter as a default value, we define $h:\mathbb N\to X$ by
\begin{equation*}
h(i):=\begin{cases}
i & \text{if $i$ codes an element of~$\mathcal T_Z^n\subseteq X$},\\
\top & \text{otherwise}.
\end{cases}
\end{equation*}
For $i,j\in\mathcal T_Z^n$ it is then trivial that $i<_{\kb(\mathbb N)}j$ implies $h(i)<_X h(j)$. One can conclude that $g\times h:\mathbb N\to 2\times X$ is a branch of~$D_\psi^n(X)$. Since $X$ is well founded while $D_\psi^n(X)$ is not, the predilator $D_\psi^n$ fails to be a dilator. To establish the converse, consider some well order~$X$ such that $D_\psi^n(X)$ is ill founded. We can then consider a branch $g\times h$ in $D_\psi^n(X)$. Let $Z\subseteq\mathbb N$ be the set with characteristic function~$g$. To conclude that $\psi(n)$ fails, we argue that $\mathcal T_Z^n$ is well founded because $h$ restricts to an embedding of~$\mathcal T_Z^n$ into the well order~$X$: Given sequences $i,j\in\mathcal T_Z^n$, we observe that $i,j\in\mathcal T_{g[m]}^n$ holds for sufficiently large~$m\in\mathbb N$ (above the length of~$i$ and $j$). Then $g[m]\times h[m]\in D_\psi^n(X)$ ensures that $i<_{\kb(\mathbb N)}j$ implies $h(i)<_Xh(j)$, as desired.
\end{proof}

Next, we construct a family of $2$-preptykes~$P_\psi^n$ such that $P_\psi^n$ is a $2$-ptyx (i.\,e.,~preserves dilators) precisely when the implication $\psi(n)\to\psi(n+1)$ holds. The construction is inspired by the one from Definition~\ref{def:dils-D_n} (which is due to Normann). It will be somewhat more technical, because we work at a higher type level; on the other hand, the fact that we are concerned with an implication between \mbox{$\Pi^1_2$-}statements (and not with a general $\Pi^1_3$-statement) allows for some simplifications.

Let us recall that the component $t\in X^{<\omega}$ of an element $s\times t\in D_\psi^n(X)$ encodes a partial embedding of $\mathcal T_Z^n$ (or rather of $\mathcal T_s^n$) into~$X$. To discuss partial embeddings on the next type level, we need some notation: For a predilator~$D$ we write
\begin{equation*}
\Sigma D:=\{(m,\sigma)\,|\,m\in\mathbb N\text{ and }\sigma\in D(m)\},
\end{equation*}
for $m=\{0,\dots,m-1\}$ with the usual linear order. To get an order on $\Sigma D$, we put
\begin{equation*}
(m,\sigma)<_{\Sigma D}(k,\tau)\quad:\Leftrightarrow\quad m<k\text{ or }(m=k\text{ and }\sigma<_{D(m)}\tau).
\end{equation*}
As before, we write $\Sigma D+1$ for the extension of $\Sigma D$ by a new maximal element $\top$ (which will, once again, serve as a default value).

\begin{definition}\label{def:partial-morph}
Consider predilators~$D$ and $E$. By a partial morphism $r:D\xRightarrow{p}E$ we mean a sequence $r=\langle r(0),\dots,r(\len(r)-1)\rangle\in(\Sigma E+1)^{<\omega}$ for which the following properties are satisfied:
\begin{enumerate}[label=(\roman*)]
\item Whenever $i<\len(r)$ is (the numerical code of) an element $(m,\sigma)\in\Sigma D$, we have $r(i)=(m,\rho)$ for some $\rho\in E(m)$, with the same first component~$m$.
\end{enumerate}
To formulate the other conditions, we assume that~(i) holds. For $(m,\sigma)\in\Sigma D$ with code $i<\len(r)$, we then define $\nu^r_m(\sigma)\in E(m)$ by stipulating $r(i)=(m,\nu^r_m(\sigma))$. We say that $\nu^r_m(\sigma)$ is undefined when $(m,\sigma)$ does not lie in $\Sigma D$ or has code~$i\geq\len(r)$.
\begin{enumerate}[label=(\roman*)]\setcounter{enumi}{1}
\item When the values $\nu^r_m(\sigma)$ and $\nu^r_m(\tau)$ are defined, we demand
\begin{equation*}
\sigma<_{D(m)}\tau\quad\Rightarrow\quad\nu^r_m(\sigma)<_{E(m)}\nu^r_m(\tau).
\end{equation*}
\item If $f:m=\{0,\dots,m-1\}\to \{0,\dots,k-1\}=k$ is an embedding, we require
\begin{equation*}
E(f)(\nu^r_m(\sigma))=\nu^r_k(D(f)(\sigma))
\end{equation*}
whenever the values $\nu^r_m(\sigma)$ and $\nu^r_k(D(f)(\sigma))$ are defined.
\end{enumerate}
\end{definition}

To avoid confusion, we explicitly state that both $r(i)\in\Sigma E$ and $r(i)=\top$ is permitted when $i<\len(r)$ does not code an element of~$\Sigma D$. Let us observe that condition~(iii) is void for all but finitely many functions $f$, since there are only finitely many numbers~$m\in\mathbb N$ such that $\nu^r_m(\sigma)$ is defined for some $\sigma\in D(m)$. As a consequence, the notion of partial morphism is $\Delta^0_1$-definable. The following is straightforward but important:

\begin{lemma}
Consider a partial morphism $r:D\xRightarrow{p}E$. For $k\leq\len(r)$, the initial segment $r[k]:D\xRightarrow{p}E$ is also a partial morphism, and we have $\nu^{r[k]}_m(\sigma)=\nu^r_m(\sigma)$ whenever $\nu^{r[k]}_m(\sigma)$ is defined (note that $\nu^{r}_m(\sigma)$ is also defined in this case).
\end{lemma}

In order to define the preptyx~$P_\psi^n$, we must specify an endofunctor of predilators and a natural family of support functions. The following definition explains the action of the endofunctor~$P_\psi^n$ on objects, i.\,e., on~predilators.

\begin{definition}\label{def:Pn_objects}
Consider a predilator~$E$. For each order~$X$, we define $P_\psi^n(E)(X)$ as the tree of all sequences $r\times s\times t\in ((\Sigma E+1)\times 2\times X)^{<\omega}$ such that $r:D_\psi^n\xRightarrow{p} E$ is a partial morphism and we have $s\times t\in D_\psi^{n+1}(X)$. Given an embedding $f:X\to Y$, we define $P_\psi^n(E)(f):P_\psi^n(E)(X)\to P_\psi^n(E)(Y)$ by
\begin{equation*}
P_\psi^n(E)(f)(r\times s\times t):=r\times D_\psi^{n+1}(f)(s\times t).
\end{equation*}
To define functions $\supp^{n,E}_X:P_\psi^n(E)(X)\to[X]^{<\omega}$, we set
\begin{equation*}
\supp^{n,E}_X(r\times s\times t):=\supp^{n+1}_X(s\times t),
\end{equation*}
where $\supp^{n+1}_X:D_\psi^{n+1}(X)\to[X]^{<\omega}$ is the support function from Definition~\ref{def:dils-D_n}.
\end{definition}

We recall that the arguments and values of our preptykes are coded predilators in the sense of Definition~\ref{def:coded-predil}, rather than class-sized predilators in the sense of Definition~\ref{def:dilator}. This is important for foundational reasons but has few practical implications, as the two variants of predilators form equivalent categories (see the first part of Section~\ref{sect:cat-dilators}). In Definition~\ref{def:dils-D_n} we have specified $D_\psi^n(X)$ for an arbitrary order~$X$, which means that we have defined~$D_\psi^n$ as a class-sized predilator. We will also write $D_\psi^n$ (rather than $D_\psi^n\!\restriction\!\lo_0$) for the coded restriction given by Lemma~\ref{lem:class-to-coded-predil}.

Assuming that $P_\psi^n$ preserves dilators, part~(b) of the following proposition entails that $D_\psi^{n+1}$ is a dilator if the same holds for~$D_\psi^n$. In view of Proposition~\ref{prop:D_n-psi}, this amounts to the implication $\psi(n)\to\psi(n+1)$.

\begin{proposition}\label{prop:P-n_objects}
The following holds for any $n\in\mathbb N$:
\begin{enumerate}[label=(\alph*)]
\item if $E$ is a predilator, then so is $P_\psi^n(E)$,
\item there is a morphism $\zeta^n:D_\psi^{n+1}\To P_\psi^n(D_\psi^n)$ of predilators.
\end{enumerate}
\end{proposition}
\begin{proof}
It is straightforward to check part~(a), based on the corresponding result from Proposition~\ref{prop:D_n-psi}. To establish part~(b), we define $g:\mathbb N\to\Sigma D_\psi^n+1$ by
\begin{equation*}
g(i):=\begin{cases}
(m,\sigma) & \text{if $i$ is the numerical code of~$(m,\sigma)\in\Sigma D_\psi^n$},\\
\top & \text{if $i$ does not code an element of $\Sigma D_\psi^n$}.
\end{cases}
\end{equation*} 
Then $g[k]:D_\psi^n\xRightarrow{p}D_\psi^n$ is a partial morphism for each $k\in\mathbb N$ (with $\nu^{g[k]}_m(\sigma)=\sigma$ when\-ever the value is defined). For each linear order~$X$, we can thus define a function $\zeta^n_X:D_\psi^{n+1}(X)\to P_\psi^n(D_\psi^n)(X)$ by setting
\begin{equation*}
\zeta^n_X(s\times t):=g[k]\times s\times t\quad\text{for $k:=\len(s)=\len(t)$}.
\end{equation*}
It is straightforward to verify that this yields a natural family of order embeddings, i.\,e.,~a morphism of predilators.
\end{proof}

In order to extend $P^n_\psi$ into a functor, we need to define its action on a (total) morphism $\mu:E_0\To E_1$. For $r\times s\times t\in P^n_\psi(E_0)(X)$ with $r:D_\psi^n\xRightarrow{p} E_0$, the first component of $P_\psi^n(\mu)_X(r\times s\times t)\in P_\psi^n(E_1)(X)$ should be a partial morphism from $D_\psi^n$ to~$E_1$. To obtain such a morphism, we compose $r$ and $\mu$ in the following way.

\begin{definition}\label{def:compose-partial-total}
Consider a (total) morphism $\mu:E_0\To E_1$ of predilators and a sequence $r=\langle r(0),\dots,r(k-1)\rangle\in(\Sigma E_0+1)^{<\omega}$. For $i<k$ we define
\begin{equation*}
\mu\circ r(i):=\begin{cases}
(m,\mu_m(\tau)) & \text{if $r(i)=(m,\tau)\in\Sigma E_0$},\\
\top & \text{if $r(i)=\top$}.
\end{cases}
\end{equation*}
We then set $\mu\circ r:=\langle \mu\circ r(0),\dots,\mu\circ r(k-1)\rangle\in(\Sigma E_1+1)^{<\omega}$.
\end{definition}

Let us verify basic properties of our construction:

\begin{lemma}\label{lem:compose-partial-morphs}
Consider a predilator~$D$, a (total) morphism $\mu:E_0\To E_1$ of predilators, and a sequence $r\in(\Sigma E_0+1)^{<\omega}$. Then $r:D\xRightarrow{p} E_0$ is a partial morphism if, and only if, the same holds for $\mu\circ r:D\xRightarrow{p} E_1$. Assuming that these equivalent statements hold, we have
\begin{equation*}
\nu_m^{\mu\circ r}(\sigma)=\mu_m(\nu_m^r(\sigma))
\end{equation*}
whenever the values $\nu_m^r(\sigma)$ and $\nu_m^{\mu\circ r}(\sigma)$ are defined.
\end{lemma}
\begin{proof}
Invoking the definition of $\mu\circ r$, it is straightforward to see that condition~(i) of Definition~\ref{def:partial-morph} holds for~$r$ if, and only if, it holds for~$\mu\circ r$. In the following we assume that~$r$ and~$\mu\circ r$ do indeed satisfy condition~(i). We can then consider the values $\nu_m^r(\sigma)$ and $\nu_m^{\mu\circ r}(\sigma)$; note that they are defined for the same arguments, namely for $(m,\sigma)\in\Sigma D$ with code below $\len(r)=\len(\mu\circ r)$. The definition of $\mu\circ r$ does also reveal that $\nu_m^{\mu\circ r}(\sigma)=\mu_m(\nu_m^r(\sigma))$ holds whenever the relevant values are defined. Since the components $\mu_m:E_0(m)\to E_1(m)$ are order embeddings, it is straightforward to conclude that condition~(ii) of Definition~\ref{def:partial-morph} holds for~$r$ if, and only if, it holds for~$\mu\circ r$. In order to establish the same for condition~(iii), we consider an embedding $f:m\to k$ and assume that the values $\nu^{r}_m(\sigma)$ and $\nu^{r}_k(D(f)(\sigma))$ are defined. If condition~(iii) holds for~$r$, we can use the naturality of~$\mu$ to get
\begin{multline*}
\nu^{\mu\circ r}_k(D(f)(\sigma))=\mu_k(\nu^r_k(D(f)(\sigma)))=\mu_k(E_0(f)(\nu^r_m(\sigma)))=\\
=E_1(f)(\mu_m(\nu^r_m(\sigma)))=E_1(f)(\nu^{\mu\circ r}_m(\sigma)),
\end{multline*}
as needed to show that condition~(iii) holds for~$\mu\circ r$. Assuming the latter, the same equalities show $\mu_k(\nu^r_k(D(f)(\sigma)))=\mu_k(E_0(f)(\nu^r_m(\sigma)))$. Since $\mu_k$ is injective, we can conclude $\nu^r_k(D(f)(\sigma))=E_0(f)(\nu^r_m(\sigma))$, as required by condition~(iii) for~$r$.
\end{proof}

To turn~$P_\psi^n$ into a preptyx, we also need to define support functions
\begin{equation*}
\Supp^n_{E,X}:P_\psi^n(E)(X)\to[\tr(E)]^{<\omega}
\end{equation*}
as in Lemmas~\ref{lem:ptyx-support-variant} and~\ref{lem:ptyx-support-variant-back}. The support of an element $r\times s\times t\in P_\psi^n(E)(X)$ will depend on the entries $r(i)\in\Sigma E+1$ with~$r(i)\neq\top$ and hence $r(i)=(m,\tau)\in\Sigma E$ with $m\in\mathbb N$ and $\tau\in E(m)$. By the proof of Theorem~\ref{thm:class-coded-equiv}, we have a unique normal form $\tau\nf E(\iota_a^m\circ\en_a)(\tau_0)$ with $a\subseteq m=\{0,\dots,m-1\}$ and $(|a|,\tau_0)\in\tr(E)$.

\begin{definition}\label{def:Pn-ptyx}
Consider a (total) morphism $\mu:E_0\To E_1$ of predilators. For each linear order~$X$, we define a function $P_\psi^n(\mu)_X:P_\psi^n(E_0)(X)\to P_\psi^n(E_1)(X)$ by
\begin{equation*}
P_\psi^n(\mu)_X(r\times s\times t):=(\mu\circ r)\times s\times t.
\end{equation*}
To define a family of support functions, we set
\begin{multline*}
\Supp^n_{E,X}(r\times s\times t)=\{(|a|,\tau_0)\,|\,\text{there is an $i<\len(r)$ with}\\
\text{$r(i)=(m,\tau)\in\Sigma E$ and $\tau\nf E(\iota_a^m\circ\en_a)(\tau_0)$}\}
\end{multline*}
for each order~$X$ and each element $r\times s\times t\in P_\psi^n(E)(X)$.
\end{definition}

To complete our construction of~$P_\psi^n$, we establish the following properties, which have been promised above.

\begin{theorem}\label{thm:ptykes-ind-step}
For each $n\in\mathbb N$, we have the following:
\begin{enumerate}[label=(\alph*)]
\item the constructions from Definitions~\ref{def:Pn_objects} and~\ref{def:Pn-ptyx} yield a preptyx $P_\psi^n$,
\item the preptyx $P_\psi^n$ is a ptyx if, and only if, we have $\psi(n)\to\psi(n+1)$.
\end{enumerate}
\end{theorem}
\begin{proof}
(a) From Proposition~\ref{prop:P-n_objects} we know that $P_\psi^n$ maps predilators to predilators. Let us now consider its action on a morphism $\mu:E_0\To E_1$. Since each component $\mu_m:E_0(m)\to E_1(m)$ is an order embedding, the same holds for the map
\begin{equation*}
\Sigma E_0\ni(m,\sigma)\mapsto(m,\mu_m(\sigma))\in\Sigma E_1.
\end{equation*}
In view of Definition~\ref{def:compose-partial-total}, one readily infers that the components $P_\psi^n(\mu)_X$ are order embeddings as well (recall that $P_\psi^n(E_i)(X)\subseteq((\Sigma E_i+1)\times 2\times X)^{<\omega}$ carries the Kleene-Brouwer order with respect to the lexicographic order on the product). To conclude that $P_\psi^n(\mu)$ is a morphism of predilators, we consider an order embedding $f:X\to Y$ and verify the naturality property
\begin{multline*}
P_\psi^n(\mu)_Y\circ P_\psi^n(E_0)(f)(r\times s\times t)=P_\psi^n(\mu)_Y(r\times D_\psi^{n+1}(f)(s\times t))=\\
=(\mu\circ r)\times D_\psi^{n+1}(f)(s\times t)=P_\psi^n(E_1)(f)((\mu\circ r)\times s\times t)=\\
=P_\psi^n(E_1)(f)\circ P_\psi^n(\mu)_X(r\times s\times t).
\end{multline*}
The claim that $P_\psi^n$ is a functor reduces to $\mu^1\circ(\mu^0\circ r)=(\mu^1\circ\mu^0)\circ r$ (for partial and total morphisms of suitable (co-)domain), which is readily verified.   To show that $P_\psi^n$ is a preptyx, we use the criterion from Lemma~\ref{lem:ptyx-support-variant-back}. Let us first verify the naturality condition
\begin{equation*}
\Supp^n_{E_1,X}\circ P_\psi^n(\mu)_X(r\times s\times t)=[\tr(\mu)]^{<\omega}\circ\Supp^n_{E_0,X}(r\times s\times t)
\end{equation*}
with respect to a morphism $\mu:E_0\To E_1$. For an arbitrary element $\tr(\mu)(|a|,\tau_0)$ of the right side, there is an index $i<\len(r)$ such that~$r(i)$ has the form $(m,\tau)\in\Sigma E_0$ with $\tau\nf E_0(\iota_a^m\circ\en_a)(\tau_0)$. We observe
\begin{equation*}
\mu_m(\tau)=\mu_m(E_0(\iota_a^m\circ\en_a)(\tau_0))=E_1(\iota_a^m\circ\en_a)(\mu_{|a|}(\tau_0)).
\end{equation*}
In view of $(|a|,\mu_{|a|}(\tau_0))=\tr(\mu)(|a|,\tau_0)\in\tr(E_1)$ (cf.~Definition~\ref{def:trace-morphism}, which relies on Lemma~\ref{lem:transfos-respect-supp}), this equation yields the normal form of $\mu_m(\tau)\in E_1(m)$. Together with $\mu\circ r(i)=(m,\mu_m(\tau))$, it follows that $\tr(\mu)(|a|,\tau_0)$ lies in the set
\begin{equation*}
\Supp^n_{E_1,X}((\mu\circ r)\times s\times t)=\Supp^n_{E_1,X}\circ P_\psi^n(\mu)_X(r\times s\times t).
\end{equation*}
Now consider any element $(|b|,\rho_0)$ of the latter. For some $i<\len(\mu\circ r)=\len(r)$, we have~$\mu\circ r(i)=(m,\rho)\in\Sigma E_1$ with $\rho\nf E_1(\iota_b^m\circ\en_b)(\rho_0)$. This yields $r(i)=(m,\tau)$ with $\mu_m(\tau)=\rho$. Writing $\tau\nf E_0(\iota_a^m\circ\en_a)(\tau_0)$, we get
\begin{equation*}
\rho=\mu_m(\tau)\nf E_1(\iota_a^m\circ\en_a)(\mu_{|a|}(\tau_0))
\end{equation*}
as above. The uniqueness of normal forms (see the proof of Theorem~\ref{thm:class-coded-equiv}) entails that we have $(b,\rho_0)=(a,\mu_{|a|}(\tau_0))$. We can conclude
\begin{equation*}
(|b|,\rho_0)=(|a|,\mu_{|a|}(\tau_0))=\tr(\mu)(|a|,\tau_0)\in [\tr(\mu)]^{<\omega}\circ\Supp^n_{E_0,X}(r\times s\times t).
\end{equation*}
The naturality property $\Supp^n_{E,Y}\circ P_\psi^n(E)(f)=\Supp^n_{E,X}$ with respect to an order embedding $f:X\to Y$ is straightforward (and in fact automatic, cf.~the discussion after Lemma~\ref{lem:ptyx-support-variant-back}). It remains to verify the support condition
\begin{equation*}
\{r\times s\times t\in P_\psi^n(E_1)(X)\,|\,\Supp_{E_1,X}^n(r\times s\times t)\subseteq\rng(\mu)\}\subseteq\rng(P_\psi^n(\mu)_X)
\end{equation*}
for a morphism $\mu:E_0\To E_1$. Given an arbitrary element $r\times s\times t$ of the left side, we construct $r'\in(\Sigma E_0+1)^{<\omega}$ with $\len(r')=\len(r)$ as follows: For each $i<\len(r)$ with $r(i)=\top$, we set $r'(i):=\top$. In the case of $r(i)=(m,\rho)\in\Sigma E_1$, we consider the unique normal form $\rho\nf E_1(\iota_a^m\circ\en_a)(\rho_0)$ and observe
\begin{equation*}
(|a|,\rho_0)\in\Supp_{E_1,X}^n(r\times s\times t)\subseteq\rng(\mu).
\end{equation*}
This allows us to write $\rho_0=\mu_{|a|}(\tau_0)$, where $\tau_0\in E_0(|a|)$ is unique since $\mu_{|a|}$ is an embedding. We now set $\tau:=E_0(\iota_a^m\circ\en_a)(\tau_0)\in E_0(m)$ and $r'(i):=(m,\tau)\in\Sigma E_0$. Observe that we have $\mu_m(\tau)=\rho$ and hence $\mu\circ r'=r$. Crucially, the equivalence from Lemma~\ref{lem:compose-partial-morphs} ensures that $r':D_\psi^n\xRightarrow{p} E_0$ is a partial morphism. This entails that we have $r'\times s\times t\in P_\psi^n(E_0)(X)$, so that we indeed get
\begin{equation*}
r\times s\times t=(\mu\circ r')\times s\times t=P_\psi^n(\mu)_X(r'\times s\times t)\in\rng(P_\psi^n(\mu)_X).
\end{equation*}

(b) For the first direction, we assume that $P_\psi^n$ is a ptyx and that $\psi(n)$ holds. By Proposition~\ref{prop:D_n-psi}, the latter entails that $D_\psi^n$ is a dilator. Since $P_\psi^n$ is a ptyx, it follows that $P_\psi^n(D_\psi^n)$ is a dilator as well. Invoking the morphism from Proposition~\ref{prop:P-n_objects}, we can conclude that $D_\psi^{n+1}$ is a dilator: For each well order~$X$, the component
\begin{equation*}
\zeta^n_X:D_\psi^{n+1}(X)\to P_\psi^n(D_\psi^n)(X)
\end{equation*}
is an order embedding, which ensures that~$D_\psi^{n+1}(X)$ is well founded. By the other direction of Proposition~\ref{prop:D_n-psi}, it follows that $\psi(n+1)$ holds, as required. To be precise, one needs to switch back and forth between class-sized and coded dilators; this is straightforward with the help of Theorem~\ref{thm:class-coded-equiv}, Corollary~\ref{cor:class-coded-dil} and Lemma~\ref{lem:morphism-dils-extend}. To establish the other direction, we assume that $\psi(n)$ holds while $\psi(n+1)$ fails. Then $D_\psi^n$ is a dilator while $D_\psi^{n+1}(X)$ is ill founded for some well order~$X$. To conclude that $P_\psi^n$ fails to be a ptyx, it suffices to show that $P_\psi^n(D_\psi^n)(X)$ is ill founded. The proof of Proposition~\ref{prop:D_n-psi} provides a function $g:\mathbb N\to\Sigma D_\psi^n+1$ such that $g[k]:D_\psi^n\xRightarrow{p} D_\psi^n$ is a partial morphism for each~$k\in\mathbb N$. We can also pick a branch $h_0\times h_1:\mathbb N\to 2\times X$ of the ill founded tree $D_\psi^{n+1}(X)\subseteq (2\times X)^{<\omega}$. It follows that $g\times h_0\times h_1$ is a branch of the tree $P_\psi^n(D_\psi^n)(X)\subseteq((\Sigma D_\psi^n+1)\times 2\times X)^{<\omega}$, which is thus ill founded.
\end{proof}

As a final ingredient for the analysis of $\Pi^1_2$-induction, we discuss the pointwise sum of ptykes. Consider a family of preptykes~$P_z$, indexed by the elements of a linear order~$Z$. We define a preptyx~$P:=\sum_{z\in Z}P_z$ as follows: Given a predilator~$E$ and a linear order~$X$, we set
\begin{gather*}
P(E)(X):=\{(z,\sigma)\,|\,z\in Z\text{ and }\sigma\in P_z(E)(X)\},\\
(z,\sigma)<_{P(E)(X)}(z',\sigma')\quad:\Leftrightarrow\quad z<_Zz'\text{ or }(z=z'\text{ and }\sigma<_{P_z(E)(X)}\sigma').
\end{gather*}
For an embedding $f:X\to Y$, we define $P(E)(f):P(E)(X)\to P(E)(Y)$ by
\begin{equation*}
P(E)(f)(z,\sigma):=(z,P_z(E)(f)(\sigma)).
\end{equation*}
The support functions $\supp^z_X:P_z(E)(X)\to[X]^{<\omega}$ of the predilators~$P_z(E)$ can be combined into a function $\supp_X:P(E)(X)\to[X]^{<\omega}$ with
\begin{equation*}
\supp_X(z,\sigma):=\supp^z_X(\sigma).
\end{equation*}
It is straightforward to check that this defines~$P(E)$ as a predilator. To turn $P$ into a functor, consider a morphism $\mu:E_0\To E_1$ and define $P(\mu):P(E_0)\To P(E_1)$ by
\begin{equation*}
P(\mu)_X(z,\sigma):=(z,P_z(\mu)_X(\sigma)).
\end{equation*}
Definition~\ref{def:ptyx-support-variant} provides functions $\Supp^z_{E,X}:P_z(E)(X)\to[\tr(E)]^{<\omega}$, which we can combine into a function $\Supp_{E,X}:P(E)(X)\to[\tr(E)]^{<\omega}$ with
\begin{equation*}
\Supp_{E,X}(z,\sigma):=\Supp^z_{E,X}(\sigma).
\end{equation*}
Invoking Lemma~\ref{lem:ptyx-support-variant-back}, we see that this turns $P=\sum_{z\in Z}P_z$ into a preptyx. Now assume that $Z$ is a well order and that $P_z$ is a ptyx for each $z\in Z$. If $E$ is a dilator and $X$ is a well order, then each order~$P_z(E)(X)$ is well founded, so that the same holds for~$P(E)(X)$. Hence $P(E)$ is a dilator and $P$ is a ptyx. Let us also observe that we have a morphism $P_z(E)\To P(E)$ for each $z\in Z$ and each predilator~$E$; its components are given by $P_z(E)(X)\ni\sigma\mapsto(z,\sigma)\in P(E)(X)$. Finally, we can prove the central result of this paper. Note that the assumption of the theorem follows from the existence of a fixed point $D\cong P(D)$ that is a dilator.

\begin{theorem}[$\aca_0$]\label{thm:fixed-point-to-induction}
Assume that every normal $2$-ptyx~$P$ admits a morphism $P(D)\To D$ for some dilator~$D$. Then the principle of $\Pi^1_2$-induction along~$\mathbb N$ holds.
\end{theorem}

Let us recall that we can quantify over (coded) dilators, since the latter are represented by subsets of~$\mathbb N$ (cf.~the discussion after Definition~\ref{def:coded-predil}). The situation for ptykes is somewhat different: In the paragraph before Proposition~\ref{prop:ptykes-direct-limit}, we have sketched how to construct a universal family of preptykes. However, we have not carried out the details of this construction. For this reason, one should read the previous theorem as a schema: Given a $\Pi^1_2$-formula~$\psi$, possibly with parameters, we will specify a $\Delta^0_1$-definable family of preptykes. Assuming that every normal ptyx~$P$ from this family admits a morphism $P(D)\To D$, we will establish the induction principle for~$\psi$ and arbitrary values of the parameters.

\begin{proof}
Let $\psi(n)$ be an arbitrary $\Pi^1_2$-formula with a distinguished number variable, possibly with further parameters. We will use the previous constructions and results for this formula~$\psi$ (with respect to some fixed normal form, cf.~the second paragraph of this section). Theorem~\ref{thm:ptykes-ind-step} provides $2$-preptykes~$P^n_\psi$ for all~$n\in\mathbb N$. Let us form their pointwise sum
\begin{equation*}
P_\psi:=\textstyle\sum_{n\in\mathbb N}P^n_\psi.
\end{equation*}
As we have seen above, we have a morpism $P^n_\psi(E)\To P_\psi(E)$ for each predilator~$E$ and each number~$n\in\mathbb N$. For the predilators $D^n_\psi$ from Proposition~\ref{prop:D_n-psi}, we will denote these morphisms by
\begin{equation*}
\iota^n:P_\psi^n(D^n_\psi)\To P_\psi(D^n_\psi).
\end{equation*}
Invoking Theorem~\ref{thm:ptyx-normal} (and the constructions that precede it), we consider the normal $2$-preptyx~$P^*_\psi:=(P_\psi)^*$ and the morphisms
\begin{equation*}
\xi^n:P_\psi(D^n_\psi)+1\To P^*_\psi(D^n_\psi+1).
\end{equation*}
Theorem~\ref{thm:ptykes-ind-step} tells us that~$P^n_\psi$ corresponds to the induction step~$\psi(n)\to\psi(n+1)$. To incorporate the base of the induction, we form the $2$-preptyx~$P^+_\psi$ with
\begin{equation*}
P^+_\psi(E):=D^0_\psi+1+P^*_\psi(E).
\end{equation*}
More formally, this can be explained as the pointwise sum $P^+_\psi=\sum_{i\in 2}P_i$, where $P_1$ is $P^*_\psi$ and $P_0$ is the constant preptyx with value $D^0_\psi+1$ (note that the latter assigns support $\Supp_{E,X}(\sigma):=\emptyset\in[\tr(E)]^{<\omega}$ to any $\sigma\in P_0(E)(X)=D^0_\psi(X)+1$). Let us observe that~$P^+_\psi$ is still normal: Given that $\mu:E_0\To E_1$ is a segment, the range of each component $P^*_\psi(\mu)_X:P_\psi^*(E_0)(X)\to P_\psi^*(E_1)(X)$ is an initial segment of the order~$P_\psi^*(E_1)(X)$, due to the normality of $P_\psi^*$ (cf.~Definitions~\ref{def:segment} and~\ref{def:ptyx-normal}). For $(0,\sigma)\in P^+_\psi(E_0)(X)$ with $\sigma\in D^0_\psi(X)+1$ we have $P^+_\psi(\mu)_X(0,\sigma)=(0,\sigma)$. So
\begin{equation*}
\rng(P^+_\psi(\mu)_X)=\{(0,\sigma)\,|\,\sigma\in D^0_\psi(X)+1\}\cup\{(1,P^*_\psi(\mu)_X(\tau)\,|\,\tau\in P^*_\psi(E_0)(X)\}
\end{equation*}
is an initial segment of $P^+_\psi(E_1)(X)$, as needed to show that $P^+_\psi$ is normal. The family of preptykes~$P^+_\psi$ is $\Delta^0_1$-definable in the parameters of~$\psi$, just as all relevant objects that we have constructed in the previous sections. We will use fixed points of ptykes from this family to derive the induction principle for~$\psi$.  For this purpose we fix values of the parameters and assume
\begin{equation*}
\psi(0)\land\forall_{n\in\mathbb N}(\psi(n)\to\psi(n+1)).
\end{equation*}
By Theorem~\ref{thm:ptykes-ind-step} we can conclude that~$P^n_\psi$ is a ptyx for every~$n\in\mathbb N$. As we have seen, it follows that $P_\psi=\sum_{n\in\mathbb N}P^n_\psi$ is a ptyx as well. Then~$P^*_\psi$ is a normal ptyx, by Theorem~\ref{thm:ptyx-normal}. Given that~$\psi(0)$ holds, Proposition~\ref{prop:D_n-psi} tells us that $D^0_\psi$ is a dilator. It is straightforward to conclude that~$P^+_\psi$ is a normal $2$-ptyx. By the assumption of the theorem, we now get a dilator~$E$ that admits a morphism
\begin{equation*}
\chi:P^+_\psi(E)\To E.
\end{equation*}
In the following, we will construct morphisms
\begin{equation*}
\kappa^n:D^n_\psi+1\To E.
\end{equation*}
Given that $E$ is a dilator, these morphisms will ensure that $D^n_\psi$ is a dilator for each number~$n\in\mathbb N$ (as in the proof of part~(b) of Theorem~\ref{thm:ptykes-ind-step}). By Proposition~\ref{prop:D_n-psi} this yields $\forall_{n\in\mathbb N}\,\psi(n)$, which is the desired conclusion of induction. To construct~$\kappa^n$, we first note that the pointwise sum~$P^+_\psi$ comes with morphisms
\begin{equation*}
\pi^0:D^0_\psi+1\To P^+_\psi(E)\quad\text{and}\quad\pi^1:P^*_\psi(E)\To P^+_\psi(E).
\end{equation*}
Using the morphisms $\zeta^n:D_\psi^{n+1}\To P_\psi^n(D_\psi^n)$ from Proposition~\ref{prop:P-n_objects} (and the construction from Example~\ref{ex:ptyx-succ}), we can recursively define
\begin{align*}
\kappa^0&:=\chi\circ\pi^0,\\
\kappa^{n+1}&:=\chi\circ\pi^1\circ P^*_\psi(\kappa^n)\circ\xi^n\circ(\iota^n+1)\circ(\zeta^n+1).
\end{align*}
Working in~$\aca_0$, this construction can be implemented as an effective recursion along~$\mathbb N$ (see~\cite{freund-eff-rec} for a detailed exposition of this principle). To justify this claim, we need to verify that $\kappa^{n+1}$ is $\Delta^0_1$-definable relative to $\kappa^n\subseteq\mathbb N$ (cf.~the discussion before Lemma~\ref{def:morphism-dilators}). This brings up a somewhat subtle point: In the discussion before Theorem~\ref{thm:ptyx-normal}, we have noted that the relation $P^*_\psi(\mu)_X(\sigma)=\tau$ is $\Delta^0_1$-definable. More precisely, the relation is defined by a $\Sigma^0_1$- and a $\Pi^0_1$-formula that are equivalent when $\mu\subseteq\mathbb N$ represents a morphism of predilators. However, we may not be able to assume that the two formulas are equivalent for any parameter $\mu\subseteq\mathbb N$. To point out one potential issue, we note that it is problematic to compose functions when we do not know whether they are total. A way around this obstacle has been presented in~\cite{freund-eff-rec}: In defining $\kappa^{n+1}$, we may anticipate the result of the construction and assume that $\kappa^n$ is a morphism of predilators. This ensures that we can use the aforementioned $\Delta^0_1$-definition of the components
\begin{equation*}
P^*_\psi(\kappa^n)_X:P^*_\psi(D^n_\psi+1)(X)\to P^*_\psi(E)(X).
\end{equation*}
It also ensures that the latter are total functions. Due to this fact, the composition above provides the required $\Delta^0_1$-definition of~$\kappa^{n+1}$.
\end{proof}

\section{Constructing fixed points of $2$-ptykes}\label{sect:construct-fixed-points}

In the present section we show how to construct a predilator~$D_P\cong P(D_P)$ for any given $2$-preptyx~$P$. Using $\Pi^1_2$-induction along~$\mathbb N$, we then prove that $D_P$ is a dilator when $P$ is a normal $2$-ptyx.

We begin with some foundational considerations. Recall that the arguments and values of our preptykes are coded predilators in the sense of Definition~\ref{def:coded-predil}. If $D$ is a coded predilator, then $D(X)$ is only defined when $X$ is a finite order of the form~$m=\{0,\dots,m-1\}$; for better readability, we will nevertheless write $D(X)$ rather than~$D(m)$. Coded predilators are represented by subsets of~$\mathbb N$, as explained after Definition~\ref{def:coded-predil}. In view of ~\cite[Section~4]{freund-single-fixed-point} it seems plausible that the desired fixed point $D_P\cong P(D_P)$ can be constructed in~$\rca_0$. We will work in~$\aca_0$, as this allows for a less technical approach and ties in with the base theory of Theorem~\ref{thm:fixed-point-to-induction}. Note that the class-sized extension $\overline{D_P}$ will still be computable (relative to $D_P\subseteq\mathbb N$; cf.~the discussion after Theorem~\ref{thm:class-coded-equiv}).

The idea is to define $D_P$ as the direct limit over iterated applications of~$P$. To provide a base for the iteration, we point out that any linear order~$Z$ gives rise to a constant predilator with values $D(X):=Z$. For each embedding~$f:X\to Y$, the embedding $D(f):D(X)\to D(Y)$ is defined as the identity on~$Z$. To obtain a natural transformation $\supp:D\To[\cdot]^{<\omega}$, we need to set $\supp_X(\sigma):=\emptyset\in[X]^{<\omega}$ for every $\sigma\in D(X)$. It is straightforward to see that this satisfies the support condition from Definition~\ref{def:dilator}. Working in $\aca_0$, the following definition can be implemented as an effective recursion along~$\mathbb N$ (see~\cite{freund-eff-rec} and the proof of Theorem~\ref{thm:fixed-point-to-induction} above). This relies on the standing assumption that all our preptykes are given by $\Delta^0_1$-relations (cf.~the discussion after Definition~\ref{def:ptyx}). For the inductive proof that $D^n_P$ is a predilator, we rely on the fact that the notion of (coded) predilator is  arithmetical (in contrast to the notion of dilator, which is $\Pi^1_2$-complete).

\begin{definition}\label{def:iterated-appl-ptyx}
Given a $2$-preptyx~$P$, we define predilators~$D^n_P$ and morphisms $\nu^n:D^n_P\To D^{n+1}_P$ by recursion over~$n\in\mathbb N$: For $n=0$ we declare that $D^0_P$ is the constant dilator with value $0$ (the empty order). Assuming that $D^n_P$ is already defined, we set $D^{n+1}_P:=P(D^n_P)$. Each component of the morphism $\nu^0$ is the empty function. Recursively, we put $\nu^{n+1}:=P(\nu^n)$.
\end{definition}

To describe the direct limit over the predilators~$D^n_P$, we will use the following observation with $\nu^n:D^n_P\To D^{n+1}_P$ at the place of $\mu:D\To E$.

\begin{lemma}\label{lem:rng-morph-inv}
For a morphism $\mu:D\To E$ and an embedding $f:X\to Y$, we have
\begin{equation*}
\sigma\in\rng(\mu_X)\quad\Leftrightarrow\quad E(f)(\sigma)\in\rng(\mu_Y)
\end{equation*}
for any element $\sigma\in E(X)$.
\end{lemma}
\begin{proof}
Given $\sigma=\mu_X(\sigma_0)\in\rng(\mu_X)$, naturality yields
\begin{equation*}
E(f)(\sigma)=E(f)\circ\mu_X(\sigma_0)=\mu_Y\circ D(f)(\sigma_0)\in\rng(\mu_Y).
\end{equation*}
Now assume $E(f)(\sigma)=\mu_Y(\tau)\in\rng(\mu_Y)$. As in the proof of Lemma~\ref{lem:weakly-normal}, we can infer $\tau\in\rng(D(f))$. Writing $\tau=D(f)(\sigma_0)$, we obtain
\begin{equation*}
E(f)(\sigma)=\mu_Y(\tau)=\mu_Y\circ D(f)(\sigma_0)=E(f)\circ\mu_X(\sigma_0).
\end{equation*}
Since $E(f)$ is an embedding, this yields $\sigma=\mu_X(\sigma_0)\in\rng(\mu_X)$.
\end{proof}

For $n\leq k$ we define $\nu^{nk}:D^n_P\To D^k_P$ as $\nu^{nk}:=\nu^{k-1}\circ\dots\circ\nu^n$. In particular, the morphism $\nu^{nn}$ is the identity on~$D^n_P$. We can now define the desired limit as follows. Note that the definition of $D_P(f)$ is justified by the previous lemma.

\begin{definition}\label{def:D_P}
For each (finite) linear order~$X$ we put
\begin{gather*}
D_P(X):=\{(m,\sigma)\,|\,m\in\mathbb N\backslash\{0\}\text{ and }\sigma\in D^m_P(X)\backslash\rng(\nu^{m-1}_X)\},\\
(m,\sigma)<_{D_P(X)}(n,\tau)\quad:\Leftrightarrow\quad\nu^{mk}_X(\sigma)<_{D^k_P(X)}\nu^{nk}_X(\tau)\text{ for }k=\max\{m,n\}.
\end{gather*}
Given an embedding $f:X\to Y$, we define $D_P(f):D_P(X)\to D_P(Y)$ by
\begin{equation*}
D_P(f)(m,\sigma):=(m,D^m_P(f)(\sigma)).
\end{equation*}
Finally, we define functions $\supp_X:D_P(X)\to[X]^{<\omega}$ by setting
\begin{equation*}
\supp_X(m,\sigma):=\supp^m_X(\sigma),
\end{equation*}
where $\supp^m_X:D^m_P(X)\to[X]^{<\omega}$ is the support function of the predilator~$D^m_P$.
\end{definition}

Let us verify that we have constructed an object of the intended type:

\begin{lemma}
The constructions from Definition~\ref{def:D_P} yield a predilator~$D_P$.
\end{lemma}
\begin{proof}
To show that each value $D_P(X)$ is a linear order, we first observe that the truth of an equality $\nu^{mk}_X(\sigma)<_{D^k_P(X)}\nu^{nk}_X(\tau)$ is independent of $k\geq\max\{m,n\}$, since all functions~$\nu^l_X$ are embeddings. Based on this fact, it is straightforward to show transitivity. Linearity reduces to the claim that $\nu^{mk}_X(\sigma)=\nu^{nk}_X(\tau)$ implies~$m=n$ (and then $\sigma=\tau$, since $\nu^{mk}_X$ is an embedding). Aiming at a contradiction, we assume~$m<n$. From
\begin{equation*}
\nu^{nk}_X(\tau)=\nu^{mk}_X(\sigma)=\nu^{nk}_X\circ\nu^{n-1}_X\circ\nu^{m(n-1)}_X(\sigma)
\end{equation*}
we then get $\tau=\nu^{n-1}_X\circ\nu^{m(n-1)}_X(\sigma)\in\rng(\nu^{n-1}_X)$, which contradicts $(n,\tau)\in D_P(X)$. To see that $D_P(f)$ is an order embedding whenever the same holds for $f:X\to Y$, we note that $\nu^{mk}_X(\sigma)<_{D^k_P(X)}\nu^{nk}_X(\tau)$ implies
\begin{equation*}
\nu^{mk}_Y\circ D^m_P(f)(\sigma)=D^k_P(f)\circ\nu^{mk}_X(\sigma)<_{D^k_P(Y)} D^k_P(f)\circ\nu^{nk}_X(\tau)=\nu^{nk}_Y\circ D^n_P(\tau).
\end{equation*}
The fact that $D_P$ is a monotone endofunctor of linear orders is readily deduced from the corresponding property of the predilators~$D^m_P$. The same applies to the naturality of supports. To verify the support condition, we consider an order embedding $f:X\to Y$ and an element $(m,\tau)\in D_P(Y)$ with
\begin{equation*}
\rng(f)\supseteq\supp_Y(m,\tau)=\supp^m_Y(\tau).
\end{equation*}
By the support condition for~$D^m_P$, we get $\tau=D^m_P(f)(\sigma)$ for some $\sigma\in D^m_P(X)$. Due to $(m,\tau)\in D_P(Y)$ we have $D^m_P(f)(\sigma)\notin\rng(\nu^{m-1}_Y)$, which implies $\sigma\notin\rng(\nu^{m-1}_X)$ by Lemma~\ref{lem:rng-morph-inv}. This yields $(m,\sigma)\in D_P(X)$ and then
\begin{equation*}
(m,\tau)=D_P(f)(m,\sigma)\in\rng(D_P(f)),
\end{equation*}
as required by the support condition for~$D_P$.
\end{proof}

As promised above, we have the following:

\begin{proposition}\label{prop:fixed-point-limit}
For each $2$-preptyx~$P$, the predilator $D_P$ is a direct limit of the system of predilators~$D^m_P$ and morphisms $\nu^{mn}:D^m_P\To D^n_P$.
\end{proposition}
\begin{proof}
The promised limit should come with morphisms $\mu^k:D^k_P\To D_P$. In order to define these, we observe that any element $\sigma\in D^k_P(X)$ lies in the range of~$\nu^{kk}_X$, since the latter is the identity. Hence there is a minimal $m\leq k$ with $\sigma\in\rng(\nu^{mk}_X)$. In view of $D^0_P(X)=0$ we must have $m>0$. We get $\sigma\notin\rng(\nu^{m-1}_X)$ by the minimality of~$m$. Since $\nu^{mk}_X$ is an embedding (and hence injective), we can define a function $\mu^k_X:D^k_P(X)\to D_P(X)$ by stipulating
\begin{equation*}
\mu^k_X(\sigma):=(m,\sigma_0)\quad\text{for $\sigma=\nu^{mk}_X(\sigma_0)$ with $m\leq k$ as small as possible}.
\end{equation*}
It is immediate that $\mu^k_X$ is an order embedding and that we have $\mu^n_X\circ\nu^{mn}_X=\mu^m_X$ for $m\leq n$. To show that the functions $\mu^k_X$ are natural in~$X$, we consider an order embedding $f:X\to Y$. Consider $\sigma\in D^k_P(X)$ and write $\mu^k_X(\sigma)=(m,\sigma_0)$. We get
\begin{equation*}
D^k_P(f)(\sigma)=D^k_P(f)\circ\nu^{mk}_X(\sigma_0)=\nu^{mk}_Y\circ D^m_P(f)(\sigma_0).
\end{equation*}
Furthermore, $\sigma\notin\rng(\nu^{m-1}_X)$ entails $D^m_P(f)(\sigma)\notin\rng(\nu^{m-1}_Y)$, by Lemma~\ref{def:D_P}. This minimality property of~$m$ allows us to conclude
\begin{equation*}
\mu^k_Y\circ D^k_P(f)(\sigma)=(m,D^m_P(f)(\sigma_0))=D_P(f)(m,\sigma_0)=D_P(f)\circ\mu^k_X(\sigma).
\end{equation*}
By Proposition~\ref{prop:direct-limit-dilators}, the claim that we have a direct limit reduces to the inclusion
\begin{equation*}
\tr(D_P)\subseteq\bigcup\{\rng(\mu^n)\,|\,n\in\mathbb N\}.
\end{equation*}
An arbitrary element of the left side has the form $(m,(n,\sigma))$ with $(n,\sigma)\in D_P(m)$ and $\supp_m(n,\sigma)=m$, where $m$ denotes the finite order $m=\{0,\dots,m-1\}$. By the definition of~$D_P$ we have $\sigma\in D^n_P(m)\backslash\rng(\nu^{n-1}_m)$. In view of $\sigma=\nu^{nn}_m(\sigma)$ this yields~$\mu^n_m(\sigma)=(n,\sigma)$. As $\supp^n_m(\sigma)=\supp_m(n,\sigma)=m$ entails $(m,\sigma)\in\tr(D^n_P)$, we can invoke Definition~\ref{def:trace-morphism} to conclude
\begin{equation*}
(m,(n,\sigma))=(m,\mu^n_m(\sigma))=\tr(\mu^n)(m,\sigma)\in\rng(\mu^n),
\end{equation*}
as required to establish the inclusion above.
\end{proof}

It is now straightforward to draw the desired conclusion:

\begin{corollary}\label{cor:iso-fixed-point}
We have $D_P\cong P(D_P)$ for any $2$-preptyx~$P$.
\end{corollary}
\begin{proof}
According to Proposition~\ref{prop:ptykes-direct-limit}, the preptyx $P$ preserves direct limits. In view of Definition~\ref{def:iterated-appl-ptyx}, we can infer that $P(D_P)$ is a direct limit of the system of predilators~$D^m_P$ and morphisms $\nu^{mn}:D^m\To D^n$ for $1\leq m\leq n$. The predilator~$D_P$ is a direct limit of the same system (note that the omission of $D^0_P$ is irrelevant). The result follows since direct limits are unique up to isomorphism.
\end{proof}

As Example~\ref{ex:ptyx-succ} shows, it can happen that $P$ is a ptyx (i.\,e.,~preserves dilators) but $D_P$ is no dilator (i.\,e., does not preserve well orders). In the following we show that $D_P$ is a dilator if we also demand that $P$ is normal in the sense of Definition~\ref{def:ptyx-normal}. Combined with Theorem~\ref{thm:fixed-point-to-induction}, this yields our main result:

\begin{theorem}\label{thm:main}
The following are equivalent over~$\aca_0$:
\begin{enumerate}[label=(\roman*)]
\item the principle of $\Pi^1_2$-induction along~$\mathbb N$ holds,
\item if $P$ is a normal $2$-ptyx, then $D_P$ is a dilator,
\item any normal $2$-ptyx~$P$ admits a morphism $P(D)\To D$ for some dilator~$D$.
\end{enumerate}
\end{theorem}
We point out that the theorem asserts an equivalence between schemas (cf.~the detailed explanation after the statement of Theorem~\ref{thm:fixed-point-to-induction}).
\begin{proof}
To see that~(ii) implies~(iii), it suffices to recall that we have $P(D_P)\cong D_P$, by Corollary~\ref{cor:iso-fixed-point}. The claim that (iii) implies (i) is the result of Theorem~\ref{thm:fixed-point-to-induction}. To establish the remaining implication from~(i) to~(ii), we consider a normal \mbox{$2$-}ptyx~$P$. In view of Definition~\ref{def:coded-dilator}, being a (coded) dilator is a $\Pi^1_2$-property. Given that~$P$ is a ptyx, we can use the induction principle from~(i) to show that $D^n_P$ is a dilator for every~$n\in\mathbb N$ (cf.~Definition~\ref{def:iterated-appl-ptyx}). The morphism~$\nu^0:D^0_P\To D^1_P$ is a segment in the sense of Definition~\ref{def:segment}, since each of its components is the empty function. Given that~$P$ is normal, another (arithmetical) induction over~$\mathbb N$ shows that all morphisms $\nu^n:D^n_P\To D^{n+1}_P$ are segments. From Proposition~\ref{prop:fixed-point-limit} we know that $D_P$ is a direct limit of the predilators~$D^n_P$. Let us show that the morphisms
\begin{equation*}
\mu^n:D^n_P\To D_P
\end{equation*}
that witness this fact are segments as well. For this purpose, we consider a (finite) order~$X$ and an inequality
\begin{equation*}
\sigma<_{D_P(X)}\mu^n_X(\tau)
\end{equation*}
with $\sigma\in D_P(X)$ and $\tau\in D^n_P(\tau)$. We need to show that $\sigma$ lies in the range of~$\mu^n_X$. Since $D_P$ is a direct limit, we have $\sigma=\mu^k_X(\rho)$ for some $k\geq n$. To see this, invoke the proof of Theorem~\ref{thm:class-coded-equiv} to write $\sigma=D_P(\iota_a^X\circ\en_a)(\sigma_0)$ with $(|a|,\sigma_0)\in\tr(D_P)$. By Proposition~\ref{prop:fixed-point-limit} (in combination with Proposition~\ref{prop:direct-limit-dilators}), we have $(|a|,\sigma_0)\in\rng(\mu^m)$ for some number~$m\in\mathbb N$. In view of Definition~\ref{def:trace-morphism}, this yields $\sigma_0=\mu^m_{|a|}(\sigma_1)$ for some $\sigma_1\in D^m_P(|a|)$. By setting $k:=\max\{m,n\}$ and $\rho:=\nu^{mk}_X\circ D^m_P(\iota_a^X\circ\en_a)(\sigma_1)$, we indeed get
\begin{equation*}
\sigma=D_P(\iota_a^X\circ\en_a)\circ\mu^m_{|a|}(\sigma_1)=\mu^m_X\circ D^m_P(\iota_a^X\circ\en_a)(\sigma_1)=\mu^k_X(\rho).
\end{equation*}
The inequality above now yields
\begin{equation*}
\mu^k_X(\rho)=\sigma<_{D_P(X)}\mu^n_X(\tau)=\mu^k_X\circ\nu^{nk}_X(\tau).
\end{equation*}
Since each component of $\mu^k$ is an embedding, we get~$\rho<_{D^k_P(X)}\nu^{nk}_X(\tau)$. As a composition of segments, the morphism $\nu^{nk}=\nu^{k-1}\circ\dots\circ\nu^n$ is a segment itself. We may thus write $\rho=\nu^{nk}_X(\rho_0)$ with $\rho_0\in D^n_P(X)$. This yields
\begin{equation*}
\sigma=\mu^k_X(\rho)=\mu^k_X\circ\nu^{nk}_X(\rho_0)=\mu^n_X(\rho_0)\in\rng(\mu^n_X),
\end{equation*}
as needed to show that $\mu^n$ is a segment. To prove that $D_P$ is a (coded) dilator, we need to consider the class-sized extension~$\overline{D_P}$ (cf.~Definition~\ref{def:coded-dilator}). Aiming at a contradiction, we assume that~$X$ is a well order with a descending sequence
\begin{equation*}
(a_0,\sigma_0)>_{\overline{D_P}(X)}(a_1,\sigma_1)>_{\overline{D_P}(X)}\ldots
\end{equation*}
in~$\overline{D_P}(X)$. Since~$D_P$ is a direct limit, we get $(a_0,\sigma_0)\in\rng((\overline{\mu^n})_X)$ for some $n\in\mathbb N$, where $\overline{\mu^n}:\overline{D^n_P}\To\overline{D_P}$ is the class-sized extension from Lemma~\ref{lem:morphism-dils-extend}. To see this, recall that Definition~\ref{def:extend-coded-dil} requires $(|a_0|,\sigma_0)\in\tr(D_P)$. By Proposition~\ref{prop:fixed-point-limit} we get
\begin{equation*}
(|a_0|,\sigma_0)=\tr(\mu^n)(|a_0|,\tau_0)=(a,\mu^n_{|a_0|}(\tau_0))
\end{equation*}
for some $n\in\mathbb N$ and some $(|a_0|,\tau_0)\in\tr(D^n_P)$. The latter ensures $(a_0,\tau_0)\in\overline{D^n_P}(X)$, so that we indeed get
\begin{equation*}
(a_0,\sigma_0)=(a,\mu^n_{|a_0|}(\tau_0))=(\overline{\mu^n})_X(a_0,\tau_0)\in\rng((\overline{\mu^n})_X).
\end{equation*}
According to Lemma~\ref{lem:ext-segment}, the class-sized extension $\overline{\mu^n}$ is also a segment. For the fixed~$n$ and any $i>0$, we can thus rely on $(a_i,\sigma_i)<_{\overline{D_P}(X)}(a_0,\sigma_0)\in\rng((\overline{\mu^n})_X)$ to infer $(a_i,\sigma_i)\in\rng((\overline{\mu^n})_X)$, say $(a_i,\sigma_i)=(\overline{\mu^n})_X(a_i,\tau_i)$. Since the components of $\overline{\mu^n}$ are order embeddings, we have a descending sequence
\begin{equation*}
(a_0,\tau_0)>_{\overline{D^n_P}(X)}(a_1,\tau_1)>_{\overline{D^n_P}(X)}\ldots
\end{equation*}
in $\overline{D^n_P}(X)$. This is the desired contradiction: since $D^n_P$ is a dilator (as we have shown by $\Pi^1_2$-induction) and~$X$ is a well order, we know that $\overline{D^n_P}(X)$ is well founded.
\end{proof}

\bibliographystyle{amsplain}
\bibliography{Fixed-point-ptykes}

\end{document}